%% file: opr2018.tex
\documentclass[preprint]{imsart}
\makeatletter
\providecommand{\@LN}[2]{}
\makeatother

\usepackage{amsthm,amsmath}
\usepackage[numbers]{natbib}
\RequirePackage[colorlinks,citecolor=blue,urlcolor=blue]{hyperref}
\usepackage{geometry}                
\input{packages.tex}
\usepackage{siunitx}
\usepackage{mathtools}
\def\namedlabel#1#2{\begingroup
    #2%
    \def\@currentlabel{#2}%
    \phantomsection\label{#1}\endgroup
}

\usetikzlibrary{decorations.pathmorphing} 
\usetikzlibrary{fit}                    
\usetikzlibrary{backgrounds}    
\numberwithin{equation}{section}

\startlocaldefs

\input{def.tex}

\endlocaldefs
\begin{document}

\begin{frontmatter}


\title{Sequential sampling of junction trees for decomposable graphs}

\begin{aug}
\author{\fnms{Jimmy} \snm{Olsson}
\ead[label=e1]{jimmyol@kth.se}},
\author{\fnms{Tetyana} \snm{Pavlenko}
\ead[label=e2]{pavlenko@kth.se}}
\and
\author{\fnms{Felix~L.} \snm{Rios}
\ead[label=e3]{flrios@kth.se}}


\affiliation{KTH Royal Institute of Technology}

\address{Department of Mathematics \\
KTH Royal Institute of Technology \\
SE-100 44 Stockholm, Sweden \\
}
\end{aug}

\runauthor{J.~Olsson, T.~Pavlenko and F.~L.~Rios}
\runtitle{Sequential sampling of junction trees}
	
\input{abstract.tex}

\end{frontmatter}

\section{Introduction}
\input{intro-new-new.tex}

\section{Preliminaries}
\label{sec:preliminaries}
\input{preliminaries.tex}

\section{Expanding and collapsing junction trees}
\label{sec:expanding_a_junction_tree}
\input{junction_tree_expanders.tex}

\section{Counting the number of junction trees for an expanded decomposable graph}
\label{sec:update_mu}
\input{update_mu.tex}

\section{Applications to sequential Monte Carlo sampling}
\label{sec:smc}
\input{smc-new.tex}

\section{Numerical study}
\label{sec:simulations}
\input{simulations-new.tex}

\section{Discussion}
\label{sec:discussion}
\input{discussion.tex}

\section{Acknowledgements}
\label{sec:acknowledgements}
\input{acknowledgements.tex}

\appendix
\input{appendix.tex}

\bibliographystyle{abbrvnat}
\bibliography{allbib}
\end{document}

%% file: packages.tex
\usepackage{graphicx}
\usepackage{amssymb}
\usepackage{amsmath}
\usepackage{natbib}
\usepackage{amsthm}
\usepackage{aliascnt}
\usepackage{amsfonts}
\usepackage{epstopdf}
\usepackage[linesnumbered,
vlined]{algorithm2e}
\usepackage{xargs}
\usepackage{bbm}
\usepackage{mathtools}
\usepackage{stmaryrd}
\usepackage{ifthen}
\usepackage{float}
\usepackage{upgreek}
\usepackage{mathrsfs}
\usepackage[textwidth=4cm, textsize=footnotesize]{todonotes}
\usepackage{cleveref}
\usepackage[utf8]{inputenc} 
\usepackage{url}
\usepackage{tikz}
\usepackage{verbatim}
\usepackage{caption}
\usepackage{subcaption}
\usepackage{lmodern}
\usepackage{enumitem}

%% file: def.tex
\newcommand{\1}{\mathbbm{1}}

\newcommandx{\bkpath}[2][1=]{
    \ifthenelse{\equal{#1}{}}
        {\bar{\kernel{R}}_{}}
        {\bar{\kernel{R}}_{} \langle #1 \rangle}
}
\newcommandx{\bk}[2][1=]{
    \ifthenelse{\equal{#1}{}}
        {\kernel{R}_{#2}}
        {\kernel{R}_{#2} \langle #1 \rangle}
    }
\newcommand{\bfft}{breath-first tree traversal}
\newcommand{\bfftabbr}{breath-first tree traversal}

\newcommandx{\comb}[2][1=]{
\ifthenelse{\equal{#1}{}}
{u_{#2}}
{u_{#1 ; #2}}
}

\newcommand{\disc}{\mathsf{pr}}

\newcommandx{\df}[1][1=]{\ifthenelse{\equal{#1}{}}{\delta}{\delta_{#1}}}

\newcommand{\eqdef}{\vcentcolon=}
\newcommand{\felix}[1]{}

\newcommandx{\gennodea}[1][1=]
{\ifthenelse{\equal{#1}{}}
{y}
{y_{#1}}
}
\newcommand{\gennodeb}{{y'}}
\newcommand{\genseta}{a}
\newcommand{\gensetb}{b}

\newcommandx{\graph}[1][1=]{\ifthenelse{\equal{#1}{}}{G}{G_{#1}}}
\newcommandx{\graphedgeset}[1][1=]{\ifthenelse{\equal{#1}{}}{E}{E_{\mathsmaller{#1}}}}
\newcommandx{\graphfd}[1][1=]{\ifthenelse{\equal{#1}{}}{\mathcal{G}}{\mathcal{G}_{#1}}}
\newcommandx{\graphnodeset}[1][1=]{\ifthenelse{\equal{#1}{}}{V}{V_{\mathsmaller{#1}}}}

\newcommandx{\graphsp}[1][1=]{\ifthenelse{\equal{#1}{}}{\mathsf{G}}{\mathsf{G}_{#1}}}
\newcommandx{\graphtarg}[1][1=]{
    \ifthenelse{\equal{#1}{}}
        {\graphtargletter^\star}
        {\graphtargletter^\star \langle #1 \rangle}
}
\newcommand{\graphtargletter}{\eta}

\newcommandx{\targMCMC}[3][1=]{
\ifthenelse{\equal{#1}{}}
{\eta_{#2}^{\N, #3}}
{\eta^* \langle #2 \rangle^{\N, #3}}
}

\newcommand{\hyperparamletter}{\vartheta}
\newcommandx{\hyperparam}[2][1=]{\ifthenelse{\equal{#1}{}}{\hyperparamletter_{#2}}{\hyperparamletter'_{#2}(#1)}}

\newcommand{\jtcount}[1]{\mu(#1)}

\newcommand{\jtind}{m}
\newcommandx{\jtnode}[1][1=]{\ifthenelse{\equal{#1}{}}{c}{c_{#1}}}
\newcommand{\jtsubtreeorder}{{|\graphnodeset'|}}

\newcommand{\jtnewnode}[1]{\jtnodearbrest{#1}^\mathsmaller{+}}
\newcommand{\jtnodeind}{j}
\newcommand{\jtnodeindb}{k}
\newcommand{\jtnodeseps}[1]{z_{#1}}
\newcommand{\jtnodearb}[1]{q_{#1}}
\newcommand{\jtnodearbrest}[1]{d_{#1}}

\newcommand{\jtnodestomove}[1]{n_{#1}}

\newcommand{\jtnoderest}[1]{r_{#1}}

\newcommand{\jtpotorigins}[1]{M_{#1}} 
\newcommand{\jtsep}[2]{\sep_{#1,#2}}
\newcommand{\jtsepcount}{\nu}

\newcommand{\jtsepforest}[1]{{\mathsf F}_{#1}}
\newcommand{\jte}{junction-tree expander}
\newcommand{\jteabbr}{JTE}
\newcommand{\jtc}{junction-tree collapser}
\newcommand{\jtcabbr}{JTC}

\newcommand{\kernel}[1]{\mathbf{#1}}
\newcommandx{\hiwlocparam}[1][1=]{\ifthenelse{\equal{#1}{}}{\phi}{\phi_{#1}}}

\newcommand{\mcmc}{Markov chain Monte Carlo}
\newcommand{\mcmcabbr}{MCMC}

\newcommand{\N}{N}

\newcommand{\newseps}[1]{A(#1)}

\newcommandx{\niwnloc}[1][1=]{\ifthenelse{\equal{#1}{}}{\alpha}{\alpha_{#1}}}
\newcommandx{\niwnscalefactor}[1][1=]{\ifthenelse{\equal{#1}{}}{\beta}{\beta[#1]}}
\newcommandx{\niwiwloc}[1][1=]{\ifthenelse{\equal{#1}{}}{\Phi}{\Phi[#1]}}
\newcommandx{\niwiwdf}[1][1=]{\ifthenelse{\equal{#1}{}}{\delta}{\delta[#1]}}
\newcommandx{\niwnlocpred}[1][1=]{\ifthenelse{\equal{#1}{}}{\alpha}{\alpha_{#1}}}
\newcommandx{\niwnscalefactorpred}[1][1=]{\ifthenelse{\equal{#1}{}}{\beta}{\beta[#1]}}
\newcommandx{\niwiwlocpred}[1][1=]{\ifthenelse{\equal{#1}{}}{\Phi}{\Phi[#1]}}
\newcommandx{\niwiwdfpred}[1][1=]{\ifthenelse{\equal{#1}{}}{\delta}{\delta[#1]}}
\newcommand{\niwnmu}{{\boldsymbol \mu}}

\newcommand{\niwiwtau}{{\boldsymbol \Phi}}

\newcommandx{\niwnmustsymb}[1][1=]{\ifthenelse{\equal{#1}{}}{\niwnmu^\ast}{\niwnmu_{#1}^\ast}}
\newcommandx{\niwiwtaustsymb}[1][1=]{\ifthenelse{\equal{#1}{}}{\niwiwtau^\ast}{\niwiwtau_{#1}^\ast}}

\newcommandx{\nn}[1][1=]{
\ifthenelse{\equal{#1}{}}
{{\mathbb N}}
{{\mathbb N}_{#1}}
}
\newcommandx{\neig}[2][1=]{
\ifthenelse{\equal{#1}{}}
{\mathfrak{N}(#2)}
{\mathfrak{N}_{#1}(#2)}
}
\newcommand{\newvert}{\gennodea_{\jtind+1}}

\newcommand{\nset}{\mathbb{N}}
\newcommand{\nsetpos}{\mathbb{N}^*}

\newcommand{\p}{p}
\newcommand{\pathorder}{\ell}
\newcommand{\pathordera}{{\ell_1}}
\newcommand{\pathorderb}{{\ell_2}}
\newcommand{\pathorderc}{{\ell_3}}
\newcommandx{\partarg}[2][1=]{
    \ifthenelse{\equal{#1}{}}
    {\eta^{\N}_{#2}}
    {\eta_{#2}^{\ast, \N}}
}
\newcommandx{\partinit}[1][1=]{
       \ifthenelse{\equal{#1}{}}
       {\kappa}
       {\kappa^\ast \langle #1 \rangle}
}

\newcommand{\potset}{\boldsymbol{\wp}}
\newcommand{\prob}{{\mathbb P}}

\newcommandx{\prop}[2][1=]{
    \ifthenelse{\equal{#1}{}}
        {\kernel{\propletter}_{#2}}
        {\kernel{\propletter}_{#2} \langle #1 \rangle}
    }
\newcommandx{\proppath}[2][1=]{
\ifthenelse{\equal{#1}{}}
{\bar{\kernel{\propletter}}_{#2}}
{\bar{\kernel{\propletter}}_{#2} \langle #1 \rangle}
}
\newcommand{\propmarg}[1]{\kernel{\propletter}^{#1}}
\newcommand{\propletter}{K}

\newcommandx{\scalmat}[1][1=]{\ifthenelse{\equal{#1}{}}{\Phi}{\Phi_{#1}}}
\renewcommand{\sep}{s}
\newcommand{\seps}[1]{S{(#1)}}
\newcommand{\sepset}{S}
\newcommand{\smcabbr}{SMC}

\newcommand{\subseqind}{r}
\newcommand{\subtreecont}{\alpha}
\newcommandx{\subtreeker}[2][1=]{
\ifthenelse{\equal{#1}{}}
{\kernel{S}^{#2}}
{\kernel{S}^{#2} \langle #1 \rangle}
}

\newcommand{\subtreenonempty}{\beta}
\newcommand{\supp}{\operatorname{Supp}}
\newcommandx{\targ}[3][1=, 3=]{
   \ifthenelse{\equal{#3}{}}
   {
       \ifthenelse{\equal{#1}{}}
       {\eta_{#2}}
       {\eta^\ast \langle #2 \rangle}
   }
   {
       \ifthenelse{\equal{#1}{}}
       {\eta_{#2} \langle #3 \rangle}
       {\eta_{#2}^\ast \langle #3 \rangle}
   }
}
\newcommand{\targpath}[1]{\bar{\eta}_{#1}}

\newcommandx{\tree}[1][1=]{\ifthenelse{\equal{#1}{}}{T}{T_{\mathsmaller{#1}}}}
\newcommand{\treepart}[2]{\tau_{#1}^{#2}}
\newcommand{\treeparttd}[2]{\tilde{\tau}_{#1}^{#2}}

\newcommandx{\treerv}[1][1=]{\ifthenelse{\equal{#1}{}}{\mathcal T}{\mathcal T_{#1}}}
\newcommandx{\tr}[2][2=]{
\ifthenelse{\equal{#2}{}}
{\tilde{x}_{#1}}
{\tilde{x}_{#1}^{#2}}
}

\newcommandx{\trex}[2][2=]{
\ifthenelse{\equal{#2}{}}
{x_{#1}}
{x_{#1}^{#2}}
}

\newcommand{\trgr}{g}

\newcommandx{\trsp}[1][1=]{\ifthenelse{\equal{#1}{}}{\mathsf{T}}{\mathsf{T}_{#1}}}
\newcommandx{\trtarg}[1][1=]{\ifthenelse{\equal{#1}{}}{\eta^\ast}{\eta^\ast \langle #1 \rangle}}

\newcommandx{\untarg}[3][1=, 3=]{
   \ifthenelse{\equal{#3}{}}
   {
       \ifthenelse{\equal{#1}{}}
       {\gamma_{#2}}
       {\gamma^\ast \langle #2 \rangle}
   }
   {
       \ifthenelse{\equal{#1}{}}
       {\gamma_{#2} \langle #3 \rangle}
       {\gamma_{#2}^\ast \langle #3 \rangle}
   }
}
\newcommandx{\untrtarg}[1][1=]{
    \ifthenelse{\equal{#1}{}}
        {\gamma^\ast}
        {\gamma^\ast \langle #1 \rangle}
}

\newcommandx{\uk}[2][1=]{
\ifthenelse{\equal{#1}{}}
{\kernel{Q}_{#2}}
{\kernel{Q}_{#2} \langle #1 \rangle}
}

\newcommandx{\unif}[1][1=]{
\ifthenelse{\equal{#1}{}}
{\mathsf{Unif}}
{\mathsf{Unif}\langle #1 \rangle}
}

\newcommandx{\ungraphtarg}[1][1=]{
    \ifthenelse{\equal{#1}{}}
        {\gamma^\star}
        {\gamma^\star \langle #1 \rangle}
}

\newcommand{\wgt}[2]{\omega_{#1}^{#2}}
\newcommandx{\wgtfunc}[2][1=]{
\ifthenelse{\equal{#1}{}}
{w_{#2}}
{w_{#2} \langle #1 \rangle}
}
\newcommand{\wgtsum}[1]{\Omega_{#1}^\N}

\newcommandx{\X}[2][2=]{
\ifthenelse{\equal{#2}{}}
{X_{#1}}
{X_{#1}^{#2}}
}
\newcommandx{\x}[2][2=]{
\ifthenelse{\equal{#2}{}}
{x_{#1}}
{x_{#1}^{#2}}
}
\newcommand{\Xset}{\mathsf{X}}

\newcommandx{\y}[1][1=]{\ifthenelse{\equal{#1}{}}{y}{y_{#1}}}
\newcommandx{\Y}[1][1=]{\ifthenelse{\equal{#1}{}}{Y}{Y_{#1}}}

\newcommandx{\Z}[2][2=]{
\ifthenelse{\equal{#2}{}}
{T_{#1}}
{T_{#1}^{#2}}
}

\newtheorem{mydef}{Definition}
\newtheorem{mythm}{Theorem}
\newtheorem*{jtexpalg}{Junction tree expander}
\newtheorem*{jtcolalg}{Junction tree collapser}

\newtheorem{myex}{Example}
\newtheorem{mylem}{Lemma}
\newtheorem{myalg}{Algorithm}
\newtheorem{myprop}{Proposition}

\newaliascnt{definition}{theorem}

\aliascntresetthe{definition}

\newaliascnt{lemma}{theorem}

\aliascntresetthe{lemma}
\crefname{lemma}{lemma}{lemmas}
\Crefname{Lemma}{Lemma}{Lemmas}

%% file: abstract.tex
\begin{abstract}
The \emph{junction-tree} representation provides an attractive structural property for organizing a \emph{decomposable graph}.
In this study, we present two novel stochastic algorithms, which we call the 
\emph{\jte}  and \emph{\jtc}  for sequential sampling of junction trees for decomposable graphs.
We show that recursive application of the \jte, expanding incrementally the underlying graph with one vertex at a time, has full support on the space of junction trees with any given number of underlying vertices.
On the other hand, the \jtc\, provides a complementary operation for removing vertices in the underlying decomposable graph of a junction tree, while maintaining the junction tree property.

A direct application of our suggested algorithms is demonstrated in a \emph{sequential-Monte-Carlo setting} designed for sampling from distributions on spaces of decomposable graphs.
Numerical studies illustrate the utility of the proposed algorithms for combinatorial computations on decomposable graphs and junction trees.  
All the methods proposed in the paper are implemented in the Python library {\it trilearn}.  
\end{abstract}

%% file: intro-new-new.tex
Decomposable graphs and their junction-tree representations as auxiliary data structure have been used in various contexts; examples include computational geometry, large-scale estimation of random-graph models with local dependence, statistical inference (such as sparse covariance- and concentration-matrix computation), contingency-table analysis, probabilistic graphical models, and message passing; see \emph{e.g.} \citep{lauritzen1996, pearl1997probabilistic, eppstein2009graph}. 
\felix{Punkt 1, tillämpningar}

In recent years, decomposable graphs have become an important and popular tool for representing dependence structures in probabilistic graphical models.
In the Bayesian setting, structure learning in such models amounts to calculating a posterior distribution defined on the space of decomposable graphs in order to gain understanding of the dependence structure underlying observed data. Due to the inherent complexity of distributions defined on such spaces, this has lead to an increasing interest in Monte Carlo methods for sampling-based approximations. 
\felix{Punkt 2}
Typical sampling strategies are based on \emph{\mcmc}\, (\mcmcabbr)\, methods, especially variations of the \emph{Metropolis--Hastings algorithm}, where new graphs are proposed by means of random single-edge perturbations, and the set of possible moves generated by subjecting a given graph to such perturbations defines a neighborhood in the decomposable-graph space; see \emph{e.g}. \citep{frydenberg1989decomposition, Giudici01121999,Thomas:2009aa, Green01032013}.
\felix{Punkt 3}
 However, since the only vertices that may be connected by an edge in a (connected) decomposable graph while maintaining decomposability are those that already have a neighbour in common and the removable edges are necessarily contained in exactly one clique, operations on the edge set are inherently local. As a consequence, an MCMC sampler based on such moves will most likely suffer from poor mixing. 
 
Against this background, it is desirable to explore alternative ways of simulating decomposable graphs. In the present paper we take a different approach which, instead of altering the edge set of a graph with a fixed set of vertices, builds new graphs incrementally, starting from the empty graph and adding vertices one by one. When applied to Monte Carlo simulation, this procedure allows for independent or weakly correlated samples. \felix{Punkt 4}

As we shall see, this approach can be designed more easily if one operates on the extended junction-tree space rather than directly on the space of decomposable graphs. This is feasible also from a statistical point of view; indeed, following \citet{Green01032013}, a given probability mass function on the space of decomposable graphs can always be embedded into an extended version defined on the space of junction trees in such a way that the push-forward distribution of the extended distribution with respect to the underlying graph equals the given probability mass function on the decomposable-graph space. Thus, by running an \mcmcabbr\, sampler producing a trajectory of junction trees targeting the extended distribution, an \mcmcabbr\, trajectory targeting the original distribution is obtained as a by-product by simply extracting the underlying graphs of the trees in the former sequence.\felix{Punkt 5}

We present two novel stochastic algorithms operating on junction-tree structures:
the \emph{\jte}\, (\jteabbr, or the \emph{Christmas-tree algorithm}) and the \emph{\jtc} (\jtcabbr). The \jteabbr\, (\jtcabbr) expands (collapses) a junction tree  \(\tree\) by randomly adding (removing) one vertex to (from) the underlying decomposable graph. At the highest level, the \jteabbr\, can be described in two main steps.
In the first step, the algorithm starts by drawing, at random, a subtree \(\tree'\) of the given tree \(\tree\) using a stochastic version of a depth-first tree-traversal. 
In the second step, a new vertex  \(\gennodea\) is connected to a random subset of each of the cliques in \(\tree'\) to form a new subtree \(\tree''\), which is isomorphic to \(\tree'\). 
The edges in \(\tree'\) are then removed and each of the nodes in \(\tree''\) are connected to the nodes in \(\tree'\) to which they stem from, while maintaining the junction tree property. 
On the other hand, the \jtcabbr\, starts by selecting the unique subtree \(\tree''\) induced by a given vertex \(\gennodeb \).
The second step amounts to drawing, for each clique \(\jtnode'\) in \(\tree''\), a neighboring clique \(\jtnode\) not containing \(\gennodeb\), for which \(\jtnode'\) is substituted while maintaining the junction tree property. 
The two algorithms are complementary in the sense that the output obtained by subjecting a given tree $\tree$ to either the \jteabbr\, followed by the \jtcabbr, or, vice versa, the \jtcabbr\, followed by the \jteabbr, coincides with $\tree$ with positive probability. \felix{punkt 6}

As we shall see, the \jteabbr\, and \jtcabbr\, have two theoretical properties that are of fundamental importance in Monte Carlo simulation. First, the transition probabilities of the induced Markov kernels are available in a closed form and can be computed efficiently; second, the \jteabbr\, algorithm is able to generate, with positive probability, when applied sequentially, all junction trees with a given number of vertices in its underlying graph.

In order to illustrate their application potential, we employ jointly the \jteabbr\, and the \jtcabbr\ to construct a \emph{sequential Monte Carlo} (SMC) \emph{sampler} \cite{10.2307/3879283} sampling from more or less arbitrary distributions defined on spaces of decomposable graphs. In this construction, which relies on the above-mentioned junction-tree embedding proposed by \citet{Green01032013}, the \jtcabbr\, is used to extend the target distribution to a path space of junction trees of increasing dimension, whereas the \jteabbr\, is used to generating proposals on this new space.
\felix{punkt 6 igen}
 Using this \smcabbr\, approach, we are able to provide unbiased estimates of the numbers of decomposable graphs and junction trees for any given number of vertices. This importance-sampling approach to the combinatorics of decomposable graphs and junction trees is the first of its kind. In the follow-up paper \citep{olsson2019}, we cast further such an \smcabbr\, sampler into the framework of \emph{particle Gibbs samplers} \cite{andrieu2010particle}. The resulting MCMC algorithm, which relies heavily on on the \jteabbr\, and \jtcabbr\, derived in the present paper, allows for \emph{global} MCMC moves across the decomposable-graph space and, consequently, fast mixing.
\felix{punkt 8}

The \jteabbr\, is related to other existing approaches of generating junction trees. For instance, the algorithm for generation of junction trees presented in \citep{randchord} has similarities to ours in the sense that it expands the underlying graph incrementally in each step of the algorithm.
However, unlike our proposed \jteabbr, this algorithm is restricted to connected decomposable graphs and transition probabilities are not directly provided. A completely different strategy for decomposable-graph sampling based on  {\it tree-dependent bipartiet graphs} is presented in \citet{elmasri2017decomposable, elmasri2017sub}.
\felix{punkt 3}

The rest of this paper is structured as follows. Section~\ref{sec:preliminaries} introduces some graph notation and a short background on decomposable graphs and junction trees. (For a more detailed presentation, the reader is referred to \emph{e.g.} \citep{ChordalGraphs} or \citep{lauritzen1996}.)
Section~\ref{sec:expanding_a_junction_tree} and Section~\ref{sub:rcta} present the \jteabbr\, and the \jtcabbr\,, respectively, along with their corresponding transition probabilities. Section~\ref{sec:update_mu} provides a novel factorisation of the number of junction trees of a decomposable graph and demonstrates its computational advantage.
The application of the \jteabbr\, and the \jtcabbr\, in the framework of \smcabbr\, sampling is found in Section~\ref{sec:smc} and 
Section~\ref{sec:simulations} contains our numerical study. 
Appendix~\ref{appendix:a} contains detailed algorithm descriptions along with the proofs of lemmas and theorems stated in the paper, whereas Appendix~\ref{appendix:b} provides an algorithm, originally presented in \citep{doi:10.1198/jcgs.2009.07129}, for randomly connecting a forest into a tree. 
\felix{9}

Finally, we remark that the code used for generating the examples in the paper is contained in the Python library \emph{trilearn} available at \url{https://github.com/felixleopoldo/trilearn}.

\label{sec:introduction}

%% file: preliminaries.tex
\subsection{Graph theory}
A pair $(\graphnodeset,\graphedgeset)$ of a \emph{vertex set} $\graphnodeset$ and an \emph{edge set} $\graphedgeset$, where
\(\graphedgeset\) is a set of unordered pairs \( (\gennodea,\gennodeb)\in \graphnodeset \times \graphnodeset \) such that \(\gennodea\ne\gennodeb\),
is called an \emph{undirected graph}. 
Two vertices $\gennodea$ and $\gennodeb$ in $\graphnodeset$ are \emph{adjacent} if they are directly connected by an edge, \emph{i.e.}, $(\gennodea, \gennodeb)=(\gennodeb, \gennodea)$ belongs to $\graphedgeset$. 
The \emph{neighbors} $\neig[(\graphnodeset,\graphedgeset)]{\gennodea}$ of a vertex $\gennodea$ is the set of vertices in $\graphnodeset$ adjacent to $\gennodea$. 
A sequence $(\gennodea_j)_{j = 1}^\ell$ of distinct vertices is called an $\gennodea_1$-$\gennodea_\ell$-path, denoted by $\gennodea_1 \sim \gennodea_\ell$, if for all $\jtnodeind \in \{2, \ldots, \ell\}$, $(\gennodea_{\jtnodeind-1}, \gennodea_\jtnodeind)$ belongs to $\graphedgeset$. 
Two vertices $\gennodea$ and $\gennodeb$ are said to be \emph{connected} if there exists an $\gennodea$-$\gennodeb$-path. 
Moreover, a \emph{graph} is said to be connected if all pairs of vertices are connected.
A graph is called a \emph{tree} if there is a unique path between any pair of vertices in the graph.
A \emph{connectivity component} of a graph is a subset of vertices that are pairwise connected. 
A graph is a \emph{forest} 
if all connectivity components induce distinct trees.
Further, two graphs are said to be \emph{isomorphic} if they have the same number of vertices and equivalent edge sets when disregarding the labels of the vertices.

Now, consider a general graph $(\graphnodeset, \graphedgeset)$ which we call $\graph$. 
The \emph{order} and the \emph{size} of $\graph$ refer to the number of vertices $|\graphnodeset|$ and the number of edges $|\graphedgeset|$, respectively.
Let $\genseta$, $\gensetb$, and $\sep$ be subsets of $\graphnodeset$; then the set $\sep$ \emph{separates} $\genseta$ from $\gensetb$ if for all $\gennodea \in \genseta$ and $\gennodeb \in \gensetb$, all paths $\gennodea \sim \gennodeb$ intersect $\sep$. 
We denote this by $\genseta \perp_{\graph} \gensetb \mid \sep$. 
The graph $\graph$ is \emph{complete} if all vertices are adjacent to each other.
A graph $(\graphnodeset^\prime, \graphedgeset^\prime)$ is a \emph{subgraph} of $\graph$ if $\graphnodeset^\prime \subseteq \graphnodeset$ and $\graphedgeset^\prime \subseteq \graphedgeset$.
A \emph{subtree} is a connected subgraph of a tree.
For $\graphnodeset^\prime \subseteq \graphnodeset$, the \emph{induced subgraph} $\graph \lbrack \graphnodeset^\prime \rbrack = (\graphnodeset^\prime, \graphedgeset^\prime)$ is the subgraph of $\graph$ with vertices $\graphnodeset^\prime$ and edge set $\graphedgeset^\prime$ given by the set of edges in $\graphedgeset$ having both endpoints in $\graphnodeset^\prime$. 
A subset of $\graphnodeset$ is a \emph{complete set} if it induces a complete subgraph.
A complete subgraph is called a \emph{clique} if it is not an induced subgraph of any other complete subgraph. 

The primer interest of this paper regards \emph{decomposable graphs} and the \emph{junction-tree} representation.
\begin{mydef} \label{def:jt_prop}
A graph $\graph$ is decomposable if its cliques can be arranged in a so-called \emph{junction tree},\emph{ i.e.} a tree whose nodes are the cliques in $\graph$, and where for any pair of cliques $\jtnode$ and $\jtnode'$ in $\graph$, the intersection $\jtnode\cap \jtnode'$ is contained in each of the cliques on the unique path $\jtnode \sim \jtnode'$.
\end{mydef}
Note that a decomposable graph may have many junction-tree representations (referred to as a junction tree for the specific graph) whereas for any specific junction tree, the underlying graph is uniquely determined.
For clarity, from now on we follow \citet{Green01032013} and reserve the terms vertices and edges for the elements of $\graph$. Vertices and edges of junctions trees will be referred to as \emph{nodes} and \emph{links}, respectively. 
Each link $(\genseta, \gensetb)$ in a junction tree is associated with the intersection $\genseta \cap \gensetb$, which is referred to as a \emph{separator} and denoted by $\jtsep{\genseta}{\gensetb}$. 
The set of distinct separators in a junction tree with graph $\graph$ is denoted by $\seps{\graph}$. 
Since all junction-tree representations of a specific decomposable graph have the same set of separators, we may talk about the separators of a decomposable graph. 
In the following we consider a fixed sequence \((\gennodea_i)_{i=1}^\p,\) of vertices and denote by \(\graphsp[\p]\) the space of decomposable graphs with vertex set \(\{\gennodea_i\}_{i=1}^\p\).
The space of junction-tree representations for graphs in \(\graphsp[\p]\) is analogously denoted by \(\trsp[\p]\). 
The graph corresponding to a junction tree $\tree$ is denoted by $\trgr(\tree)$.
We let $\tree_\sep$ denote the subtree induced by the nodes of a junction tree $\tree$ containing the separator $\sep$ and let $\jtsepforest{\sep}(\tree)$ denote the forest obtained by deleting, in $\tree$, the links associated with $\sep$. 

\subsection{Notational convention}
For any finite set $\genseta$, we denote its power set by $\potset(\genseta)$.
The uniform distribution over the elements in $\genseta$ is denoted by $\unif[\genseta]$. 
We assume that all random variables are well defined on a common probability space $(\Omega, \potset{(\Omega)}, \prob)$.  
Abusing notation, we will always use the same notation for a random variable and a realisation of the same. 
Further, we will use the same notation for a distribution and its corresponding probability density function.
For an arbitrary space $\Xset$, the \emph{support} of a nonnegative  function $h$ defined on $\Xset$ is denoted by $\supp(h) \eqdef \{x \in \Xset: h(x)>0\}$.
For all sequences $(\genseta_j)_{j = 1}^\ell$, we apply the convention $(\genseta_j)_{j = 1}^0 \eqdef  \emptyset$. 
Moreover, for all sequences $(\genseta_j)_{j = 1}^\ell$ of sets and all nonempty sets $\gensetb$, we set $\gensetb \cap \left (\cup_{j = 1}^0 \genseta_j \right) \eqdef \gensetb$. 
We denote by $\mathbb{N}$ the set of natural numbers \(\{1,2,\dots\}\) and by $\mathbb{N}_{p}$ the set  \(\{1,\dots,\p\}\) for some \(\p\in \mathbb{N}\).

The notation, $\disc(\{ w_\ell \}_{\ell = 1}^\N)$ is used to denote the categorical distribution induced by a set $\{ w_\ell \}_{\ell = 1}^\N$ of positive (possibly unnormalised) numbers. 
More specifically, writing $x \sim \disc(\{ w_\ell \}_{\ell = 1}^\N)$ means that the random variable $x$ takes on the value $\ell \in \nn[\N]$ with probability $\textstyle w_\ell / \sum_{\ell' = 1}^\N w_{\ell'}$.

%% file: junction_tree_expanders.tex

\subsection{Sampling subtrees} 
\label{sub:random_sampling_of_subtrees}

\input{random_subtree.tex}

\subsection{Expanding junction trees} 
\label{sub:cta}
\input{cta.tex}

\section{Collapsing junction trees}
\label{sub:rcta}
\input{rcta.tex}


%% file: random_subtree.tex
Before presenting our main algorithm for expanding junction trees, we present
one of its crucial subroutines: an algorithm for random sampling of subtrees of a given, arbitrary tree \(\tree\).
It takes two tuning parameters, $(\alpha, \beta) \in (0,1)^2$, which together control the number of vertices in the subtree. 

At the initial step of the algorithm, a Bernoulli trial with probability $\beta$ is performed in order to decide whether to return the empty subtree or not.  
If the empty subtree was not drawn we pick a vertex 
$\gennodea$ uniformly at random and let the subtree grow stochastically around this node. 
More specifically, $\gennodea$ serves as a root in a procedure similar to a \emph{\bfft}, where 
instead of directly adding non-visited neighbors of a vertex to the queue of vertices to be visited, as done in the standard \bfftabbr, each non-visited neighbor of $\gennodea$ is added with probability $\alpha$. 
Thus, the parameter $\alpha$ controls the number of vertices in the subgraph, given that it is nonempty. 
The algorithm is presented in full detail in Algorithm~\ref{alg:sbfs} in Appendix \ref{appendix:a}. 

The induced probability of extracting a subtree $\tree' = (\graphnodeset^\prime,\graphedgeset^\prime)$ from a given tree $\tree = (\graphnodeset, \graphedgeset)$ is given by 
\begin{align*}
\label{eq:prob_subtree}
    \subtreeker{\subtreecont,\subtreenonempty}(\tree, \tree^\prime)=
    \begin{cases}
        1-\subtreenonempty, & \text{ if } \tree^\prime=(\emptyset, \emptyset) \\
        \subtreenonempty|\graphnodeset^\prime|\subtreecont^{|\graphnodeset^\prime|-1}(1-\subtreecont)^w/|\graphnodeset|, &\text{ otherwise},
    \end{cases}
\end{align*}
where $w=w(\tree,\tree')$ is the number of components in the forest $\tree\lbrack \graphnodeset\setminus\graphnodeset'\rbrack$.
The factor $|\graphnodeset^\prime|$ stems from the fact that any vertex in $\graphnodeset^\prime$ is a valid root vertex 
in the breadth-first traversal-like procedure and the probability of extracting a certain subtree is the equal for each choice.

%% file: cta.tex
In this section we present the main contribution of this paper, namely an algorithm for expanding 
randomly a given junction tree $\tree \in \trsp[\jtind]$ into a new junction tree $\tree[+] \in \trsp[\jtind + 1]$ such that $\trgr({\tree})$ is the induced subgraph of $\trgr(\tree[+])$. This operation defines a Markov transition kernel $\prop{m}^{\subtreecont,\subtreenonempty}:\trsp[\jtind] \times \potset(\trsp[\jtind+1]) \rightarrow [0,1]$, whose expression is derived at the end of this section. 
Since the theorems in this section hold for any valid choice of $\alpha$ and $\beta$, we sometimes drop these from the notation and simply write $\prop{\jtind}$ instead of $\prop{\jtind}^{\alpha,\beta}$.
The full procedure, which in the following will be referred to as the \emph{\jte}, is given below.
Further details of these steps are provided in  Algorithm~\ref{alg:cta} in Appendix~\ref{appendix:a}.

\begin{jtexpalg}
Let \(\tree\) be a junction tree in \(\trsp[\jtind]\).
\setlist[enumerate]{align=left}
\begin{enumerate}
    \item[\emph{Step 1.} ] Sample a random subtree \(\tree'=(\graphnodeset', \graphedgeset')\) of \(\tree\). \label{alg:jte:step1}
\end{enumerate}

If \(\tree'\) is empty, proceed as follows:
 \begin{enumerate}
 \setlength\itemsep{1em}
    \item[\emph{Step 2.} ]Create a new node containing merely the vertex \(\newvert\) and connect it to an arbitrary node in \(\tree\).\label{alg:jte:step2}
    \item[\emph{Step 3.} ]Cut the new tree at the empty separator to obtain a forest.\label{alg:jte:step3}
    \item[\emph{Step 4.} ]Randomly reconnect the forest into a tree (see Algorithm~\ref{alg:jt_shuffling} in Appendix~\ref{appendix:b}). \label{alg:jte:step4}
\end{enumerate}

If \(\tree'\) is non-empty,
 enumerate the nodes in \(\graphnodeset'\) as  \( (\jtnode[\jtnodeind] )_{j = 1}^\jtsubtreeorder \),
    and let, for each \(\jtnodeind\),  \(\jtnodeseps{\jtnodeind}\) be defined as the union of the separators associated with \(\jtnode[\jtnodeind]\) in \(\tree'\).
Proceed as follows:
\begin{enumerate}
 \setlength\itemsep{1em}
    \item[\emph{Step 2\(^\ast\).} ] \label{alg:jte:step2*}
    For each node \(\jtnode[\jtnodeind]\), draw uniformly at random a (possibly empty) subset \(\jtnodearb{\jtnodeind}\) of \(\jtnode[\jtnodeind] \setminus \jtnodeseps{\jtnodeind}\) to
    create a new unique node \( \jtnewnode{\jtnodeind} \),  consisting of \(\jtnodearbrest{\jtnodeind} \eqdef \jtnodeseps{\jtnodeind}\cup \jtnodearb{\jtnodeind}\) and the vertex \(\newvert\).
     Note that for  \( \jtnewnode{\jtnodeind} \) to be unique, \(\jtnodearb{\jtnodeind}\) has to be non-empty if any separator  associated with \(\jtnode_{\jtnodeind}\) in \(\tree'\) equals \(\jtnodeseps{\jtnodeind}\). 
    If \(\jtnode[\jtnodeind]\) was engulfed in \(\jtnewnode{\jtnodeind}\) (i.e. $ \jtnodearbrest{\jtnodeind} = \jtnode_{\jtnodeind}$), simply delete \(\jtnode[\jtnodeind]\).
   
    \item[\emph{Step 3\(^\ast\).} ]\label{alg:jte:step3*} To the nodes in \(( \jtnewnode{\jtnodeind} )_{j = 1}^\jtsubtreeorder\), assign links which replicate the structure of  \(\tree'\). Then remove the links in \(\tree'\) and connect by an link each \(\jtnode[\jtnodeind]\) to its corresponding new node \(\jtnewnode{\jtnodeind}\).
    \item[\emph{Step 4\(^\ast\).} ] \label{alg:jte:step4*}
    For each node \(\jtnode[\jtnodeind]\), the neighbors whose links can be moved to \(\jtnewnode{\jtnodeind}\) while maintaining an equivalent junction tree, are distributed uniformly between \(\jtnode[\jtnodeind]\) and \(\jtnewnode{\jtnodeind}\).
    The set of neighbors of  \(\jtnewnode{\jtnodeind}\) is denoted by \(\jtnodestomove{\jtnodeind}\).
    
\end{enumerate}
\end{jtexpalg}

When using the subtree sampler provided in Section~\ref{alg:sbfs} at Step 1, the parameters $\subtreecont$ and $\subtreenonempty$ have clear impacts on the sparsity of the outcome $\tree[+]$ of the \jteabbr;
more specifically, since each node in the selected subtree will give rise to a new node in $\tree[+]$, $\subtreecont$ controls the number of nodes containing the new vertex $\newvert$. 
The parameter $\subtreenonempty$ is simply interpreted as the probability of $\newvert$ being connected to some vertex in $\trgr(\tree)$.

\input{ex_expanding.tex}

\begin{figure}[ht] 
    \captionsetup[subfigure]{aboveskip=-1pt,belowskip=-1pt}
    \centering
    \begin{subfigure}[t]{0.06\textwidth}
        \centering
         \input{figures/tikz/algrun/algrun_1_jt.tikz}
        \caption*{$\tree_1$}
        \label{fig:T1}
    \end{subfigure}
    \hspace{1cm}
    \begin{subfigure}[t]{0.06\textwidth}
        \centering
        \input{figures/tikz/algrun/algrun_1_graph.tikz}
        \caption*{$\graph[1]$}
        \label{fig:G1}
    \end{subfigure}

    \vspace{0.7cm}
    \begin{subfigure}[t]{0.06\textwidth}
        \centering
        \input{figures/tikz/algrun/algrun_1_jt_p.tikz}
        \caption*{$\tree_1'$}
        \label{fig:T1p}
    \end{subfigure}
    \hspace{1cm}
    \begin{subfigure}[t]{0.06\textwidth}
        \centering
        \input{figures/tikz/algrun/algrun_1_graph_p.tikz}
        \caption*{$\graph[1]'$}
        \label{fig:G1p}
    \end{subfigure}

    \begin{subfigure}[t]{0.08\textwidth}
        \centering
        \input{figures/tikz/algrun/algrun_2_jt.tikz}
        \caption*{$\tree_2$}
        \label{fig:T2}
    \end{subfigure}
    \hspace{1cm}
    \begin{subfigure}[t]{0.11\textwidth}
        \centering
        \input{figures/tikz/algrun/algrun_2_graph.tikz}
        \caption*{$\graph[2]$}
        \label{fig:G2}
    \end{subfigure}

    \vspace{0.7cm}
    \begin{subfigure}[t]{0.08\textwidth}
        \centering
        \input{figures/tikz/algrun/algrun_2_jt_p.tikz}
        \caption*{$\tree_2'$} 
        \label{fig:T2p}
    \end{subfigure}
    \hspace{1cm}
    \begin{subfigure}[t]{0.11\textwidth}
        \centering
        \input{figures/tikz/algrun/algrun_2_graph_p.tikz}
        \caption*{$\graph[2]'$}
        \label{fig:G2p}
    \end{subfigure}
    ~

    \begin{subfigure}[t]{0.16\textwidth}
        \centering
        \input{figures/tikz/algrun/algrun_3_jt.tikz}
        \caption*{$\tree_3$}
        \label{fig:T3}
    \end{subfigure}
    \hspace{1cm}
    \begin{subfigure}[t]{0.15\textwidth}
        \centering
        \input{figures/tikz/algrun/algrun_3_graph.tikz}
        \caption*{$\graph[3]$}
        \label{fig:G3}
    \end{subfigure}
    \vspace{0.7cm}

    \begin{subfigure}[t]{0.16\textwidth}
        \centering
        \input{figures/tikz/algrun/algrun_3_jt_p.tikz}
        \caption*{$\tree_3'$}
        \label{fig:T3p}
    \end{subfigure}
    \hspace{1cm}
    \begin{subfigure}[t]{0.15\textwidth}
        \centering
        \input{figures/tikz/algrun/algrun_3_graph_p.tikz}
    \hspace{1cm}
        \caption*{$\graph[3]'$}
        \label{fig:G3p}
    \end{subfigure}
    ~

    \begin{subfigure}[t]{0.26\textwidth}
        \centering
        \input{figures/tikz/algrun/algrun_4_jt.tikz}
        \caption*{$\tree_4$}
        \label{fig:T4}
    \end{subfigure}
    \hspace{1cm}
    \begin{subfigure}[t]{0.2\textwidth}
        \centering
        \input{figures/tikz/algrun/algrun_4_graph.tikz}
        \caption*{$\graph_4$}
        \label{fig:G4}
    \end{subfigure}

    \vspace{0.7cm}
    \begin{subfigure}[t]{0.26\textwidth}
        \centering
        \input{figures/tikz/algrun/algrun_4_jt_p.tikz}
        \caption*{$\tree_4'$}
        \label{fig:T4p}
    \end{subfigure}
    \hspace{1cm}
    \begin{subfigure}[t]{0.2\textwidth}
        \centering
        \input{figures/tikz/algrun/algrun_4_graph_p.tikz}
        \caption*{$\graph[4]'$}
        \label{fig:G4p}
    \end{subfigure}
    ~

    \begin{subfigure}[t]{0.26\textwidth}
        \centering
        \input{figures/tikz/algrun/algrun_5_jt.tikz}
        \caption*{$\tree_5$}
        \label{fig:T5}
    \end{subfigure}
    \hspace{1cm}
    \begin{subfigure}[t]{0.2\textwidth}
        \centering
        \input{figures/tikz/algrun/algrun_5_graph.tikz}
        \caption*{$\graph[5]$}
        \label{fig:G5}
    \end{subfigure}
    \caption{Example of a recursive application of the \jteabbr\,with parameters $\subtreecont=0.3$ and $\subtreenonempty=0.9$.}
    \label{fig:alg_run}
\end{figure}
\input{ex_cta_traj.tex}

The main reason for operating on junction trees as opposed to decomposable graphs directly is  computational tractability.
Next we provide explicit expressions of the transition kernel $\prop{\jtind}$ of the \jteabbr, for any given \(\jtind\in \nset\).   

For $\tree\in \trsp[\jtind]$ and $\tree[+]$ generated by the \jteabbr, let $\mathsf{\tree}(\tree, \tree[+])$ denote the set of possible subtrees bridging $\tree$ and $\tree[+]$ through the first step of the \jteabbr.
This set contains, depending on $\tree$ and $\tree[+]$, either one unique or two different trees, whose explicit forms are provided by Proposition \ref{thm:uniqueness}.

\begin{myprop}
\label{thm:uniqueness}
    Let $m \in \nset$, $\tree\in \trsp[\jtind]$, and $\tree[+]$ be generated by the \jteabbr.
    If the subtree of $\tree[+]$ induced by the nodes containing the vertex $ \gennodea[\jtind+1]$ has a single node $\{\gennodea[\jtind+1]\} \cup \sep$ with exactly two neighbors $\jtnode[1]$ and $\jtnode[2]$ such that $\jtsep{\jtnode[1]}{\jtnode[2]}=\sep$, then $\mathsf{\tree}(\tree, \tree[+]) =\{(\{\jtnode[1]\}, \emptyset),(\{\jtnode[2]\}, \emptyset)\}$;
    otherwise,
    $\mathsf{\tree}(\tree, \tree[+]) = \{(\graphnodeset', \graphedgeset')\}$ (a single tree), where $\graphnodeset'=\{\jtnode_{\jtnodeind} \in \graphnodeset: \jtnewnode{\jtnodeind} \in \graphnodeset[+] \}$ and $ \graphedgeset'=\{(\jtnode_{\jtnodeind}, \jtnode_{\jtnodeindb})\in \graphedgeset : (\jtnewnode{\jtnodeind},\jtnewnode{\jtnodeindb})\in \graphedgeset[+]'\}$.
    Here $\jtnewnode{\jtnodeind} := \jtnodearbrest{\jtnodeind} \cup \{\gennodea[\jtind+1]\}$ and $\jtnewnode{\jtnodeindb} := \jtnodearbrest{\jtnodeindb} \cup \{\gennodea[\jtind+1]\}$ denote new nodes in $\tree[+]$ and
    $\jtnode_{\jtnodeind}:=\jtnodearbrest{\jtnodeind} \cup \jtnoderest{\jtnodeind}$ and $\jtnode_{\jtnodeindb}:=\jtnodearbrest{\jtnodeindb} \cup \jtnoderest{\jtnodeindb}$ are the corresponding nodes in $\tree$. 
    The sets $ \jtnoderest{\jtnodeind}$ and $ \jtnoderest{\jtnodeindb}$ may be empty.
\end{myprop}
From a computational point of view, Proposition~\ref{thm:uniqueness} is crucial since it guarantees a tractable expression of $\prop{\jtind}^{\subtreecont, \subtreenonempty}$. 
Before we state this expression we introduce some further notation.
We let $\jtsepcount_\mathsmaller{\trgr(\tree)}(\sep)$ denote the number of possible ways that $\jtsepforest{\sep}(\tree)$, the tree obtained by cutting $\tree_{\sep}$ at the separator $\sep$, can be connected to form a tree; this number is described in more detail in Theorem~\ref{thm:nu_g}. 
Now, the transition probability of 
the \jteabbr\, 
takes the following form
\begin{equation} \label{eq:cta}
    \prop{\jtind}^{\subtreecont, \subtreenonempty}(\tree, \tree[+])
     = \sum_{T' \in \mathsf{T}(T, T_+)} \prob(\tree[+] \mid\tree^\prime, \tree ) \subtreeker{\subtreecont,\subtreenonempty}(\tree, \tree^\prime),  
 \end{equation}
where $\prob(\tree[+] \mid\tree^\prime, \tree)$ is understood as the probability that the \jteabbr\, generates $\tree[+]$ with $\tree$ as input given that $\tree^\prime$ was drawn at
Step~1.
We stress again that the sum in \eqref{eq:cta} has either one or two terms and it is thus easily computed. The conditional probability $\prob(\tree[+] \mid\tree^\prime, \tree)$ takes two different forms depending on whether $\tree'$ is empty or not. 
If $\tree'$ is empty, since $\tree[+]$ is randomised at $\emptyset$, all the $\jtsepcount_{\mathsmaller{\trgr({\tree[+]})}}(\emptyset)$ obtainable equivalent junction trees have equal probability. 
Otherwise, in case of $\tree'$ non-empty, the probability of the subsets $\jtnodearb{\jtnodeind}$ are calculated according to the uniform subset distributions in Step 2\(^\ast\).
Observe that, given $\tree$ and $\tree'$, the resulting tree $\tree[+]$ is completely determined by $\{\jtnodearb{\jtnodeind}\}_{\jtnodeind=1}^\jtsubtreeorder$ and $\{\jtnodestomove{\jtnodeind}\}_{\jtnodeind=1}^\jtsubtreeorder$.
Since the pairs $(\jtnodearb{\jtnodeind},\jtnodestomove{\jtnodeind})_{j = 1}^\jtsubtreeorder$ are drawn conditionally independently given $\tree'$ and $\tree$ we obtain 
\begin{equation} \label{eq:cond:prob:factorisation}
    \prob(\tree[+] \mid\tree^\prime, \tree ) = 
    \begin{cases}
        1/\jtsepcount_{\mathsmaller{\trgr({\tree[+]})}}(\emptyset) &\text{if } \tree'=(\emptyset,\emptyset),\\
        \prod_{\jtnodeind=1}^\jtsubtreeorder \prob(\jtnodearb{\jtnodeind} \mid \tree', \tree )  \prob(\jtnodestomove{\jtnodeind} \mid \jtnodearb{\jtnodeind}, \tree', \tree ) &\text{otherwise.}
    \end{cases}
\end{equation}

We examine the probabilities in \eqref{eq:cond:prob:factorisation} in the case where $\tree'$ is nonempty. Since for each $\jtnodeind$, the existence of a node $\jtnode \in \neig[\tree']{\jtnode_{\jtnodeind}}$ such that $\jtnodeseps{\jtnodeind} = \jtsep{\jtnode}{\jtnode_{\jtnodeind}}$ forces $\jtnodearb{\jtnodeind}$ to be nonempty, it holds that 
\begin{align*}
    \prob(\jtnodearb{\jtnodeind}\mid\tree^\prime,\tree) =
    \begin{cases}
        1/(2^{|\jtnode_{\jtnodeind} \setminus \jtnodeseps{\jtnodeind}|} -1 ) & \text{if } \jtnodeseps{\jtnodeind} = \jtsep{\jtnode}{\jtnode_{\jtnodeind}} \mbox{ for some }\jtnode \in \neig[\tree']{\jtnode_{\jtnodeind}}, \\
         1/2^{|\jtnode_{\jtnodeind} \setminus \jtnodeseps{\jtnodeind}|} & \text{otherwise.}
    \end{cases}
\end{align*}

Conditionally upon $\tree'$, $\tree$, and $\jtnodearb{\jtnodeind}$, the probability of each neighbor set $\jtnodestomove{\jtnodeind}$ 
at Step \(4^\ast\) 
follows straightforwardly; indeed, the distribution of $\jtnodestomove{\jtnodeind}$ takes two different forms depending on whether $\jtnode_\jtnodeind$ was engulfed into $\jtnewnode{\jtnodeind}$ (i.e. $ \jtnodearbrest{\jtnodeind} = \jtnode_{\jtnodeind}$) or not.
If so, all of the neighbors of $\jtnode_\jtnodeind$ are moved to $\jtnewnode{\jtnodeind}$ with probability 1. 
Otherwise, it has equal probability over all subsets of $U_\jtnodeind \eqdef \{\jtnode \in \neig[\tree]{\jtnode_{\jtnodeind}} \setminus \graphnodeset' : \sep_{ \jtnode, \jtnode_{\jtnodeind}} \subseteq \jtnodearbrest{\jtnodeind} \}$ giving 
$$
    \prob(\jtnodestomove{\jtnodeind}\mid\jtnodearb{\jtnodeind},\tree^\prime, \tree ) =
    \begin{cases}
        1 & \text{if }  \jtnodearbrest{\jtnodeind} = \jtnode_{\jtnodeind} , \\
        1/2^{|U_\jtnodeind|} & \text{otherwise.}
    \end{cases}
$$
Observe that the simplicity of \eqref{eq:cta} is appealing from a computational point of view. 
In particular, as shown in Section~\ref{sec:simulations}, when $\prop{\jtind}^{\subtreecont, \subtreenonempty}$ is used as a proposal kernel in an \smcabbr\, algorithm, fast computation of the transition probability is crucial, especially as the graph space increases.

An important property of the \jteabbr\, is that for any $\jtind\in \nset$ and \(\tree\in \trsp[\jtind]\), a tree \(\tree[+]\) generated by the \jteabbr\, is also a junction tree.
In addition, \(\trgr(\tree)\) is an induced subgraph of \(\trgr(\tree[+])\), having one additional vertex.
\begin{mythm}
    \label{thm:jt_preservation}
    For any $m \in \nset$ and $\tree\in \trsp[m]$ it holds that
    \begin{enumerate}[label=(\roman*)]
    \setlength\itemsep{0.5em}
        \item 
            $\supp(\prop{m}(\tree, \cdot)) \subseteq \trsp[m+1]$, \label{itm:cta:jt_preservation_prop}
        \item 
            $ \trgr(\tree[+])[\{\gennodea[\ell]\}_{\ell=1}^{\jtind}] = \trgr(\tree)$ for all $ \tree[+]\in \supp(\prop{m}(\tree,\cdot))$. \label{itm:cta:jt_ind_subgraph_prop}
    \end{enumerate}
\end{mythm}

The following theorem states that for any $\jtind \in \nset$, all junction trees in $\trsp[\jtind]$ can be generated with positive probability using recursive application of the \jteabbr.
More specifically, we may define the marginal probability
$
    \propmarg{\jtind}(\tree_\jtind) \eqdef \sum_{\tree\in \trsp[\jtind-1]}\propmarg{\jtind-1}(\tree)\prop{\jtind-1}(\tree, \tree_{\jtind}),$ for any $ \tree_\jtind\in \trsp[\jtind],
$
where
$\propmarg{1}((\{\{\gennodea[1]\}\}, \emptyset)) = 1$
and state the following theorem.
\begin{mythm}
    \label{thm:characterization}
    For any ordering of vertices \((\gennodea[\ell])_{\ell=1}^\jtind\), $m \in \nset$, it holds that 
    \begin{align*}
        \supp(\propmarg{\jtind}) = \trsp[\jtind]. 
    \end{align*}

\end{mythm}

%% file: ex_expanding.tex
\begin{myex}
\label{ex:cta_expansion}
We illustrate two possible scenarios for how the junction tree in Figure~\ref{fig:graph_and_jt} with underlying vertex set \(\graphnodeset=\{1,\dots,9\}\) could be expanded by the vertex $10$. 
Figure~\ref{fig:expanded_graph_and_jt} shows the possible scenario where the subtree picked at 
Step~1 
is empty.
Figure~\ref{fig:expanded_graph_and_jt_nonempty} demonstrates the possible scenario where the subtree sampled at 
Step~1 
contains the nodes $\jtnode[1]=\{3,4\}$, $\jtnode[2]=\{1,4,5\}$, and $\jtnode[3]=\{2,5,6\}$, colored in blue. 
The new nodes, colored in red, are $\jtnodearbrest{1}^+=\{3,4,10\}$, $\jtnodearbrest{2}^+=\{4,5,10\}$, and $\jtnodearbrest{3}^+=\{5,6,10\}$, built from the sets $\jtnodeseps{1}=\{4\}$, $\jtnodeseps{2}=\{4,5\}$, $\jtnodeseps{3}=\{5\}$ and $\jtnodearb{1}=\{3\}$, $\jtnodearb{2}=\emptyset$, $\jtnodearb{3}=\{6\}$. 
The sets of moved neighbors are $\jtnodestomove{1}=\emptyset$, $\jtnodestomove{2}=\emptyset$ and $\jtnodestomove{3}=\{\{5,6,9\}\}$.

Note that in this example, \(\jtnode_3\) is a leaf node in the resulting tree, making it look like decoration in a Christmas tree.
 \begin{figure}
    \centering
    \begin{subfigure}[b]{0.35\textwidth}
        \resizebox{5cm}{!}{
            \input{figures/tikz/jtexp1_graph.tikz}
        }
    \end{subfigure}
    ~\hspace{7mm}%
    \begin{subfigure}[t]{0.35\textwidth}
        \resizebox{5cm}{!}{
            \input{figures/tikz/jtexp.tikz}
        }
    \end{subfigure}

    \caption{A decomposable graph (left panel) and a corresponding junction tree representation (right panel).}
        \label{fig:graph_and_jt}
\end{figure}
\captionsetup[subfigure]{labelformat=empty, labelsep=none}
 \begin{figure}
    \centering
    \begin{subfigure}[t]{0.35\textwidth}
        \resizebox{5cm}{!}{
            \input{figures/tikz/exp_jt_isolated.tikz}
        }
        \subcaption{\emph{Step~2}: The new node, $\{10\}$ is connected to an arbitrary existing node.}
        \label{fig:expanded_graph_and_jt:a}
    \end{subfigure}
~\hspace{7mm}%
    \begin{subfigure}[t]{0.35\textwidth}
        \resizebox{5cm}{!}{
            \input{figures/tikz/forest_emptyset.tikz}
        }
        \caption{\emph{Step~3}: Create forest.}
        \label{fig:expanded_graph_and_jt:b}
    \end{subfigure}

    \begin{subfigure}[t]{0.35\textwidth}
        \resizebox{5cm}{!}{
            \input{figures/tikz/exp_jt_isolated_1.tikz}
        }
        \subcaption{\emph{Step~4}:
        Reconnect 
        the forest into a tree.}
        \label{fig:expanded_graph_and_jt:c}
    \end{subfigure}
    \caption{A possible expansion of the junction tree in Figure~\ref{fig:graph_and_jt}, where the empty subtree was drawn 
    at Step 1.
}
    \label{fig:expanded_graph_and_jt}

\end{figure}
\captionsetup[subfigure]{labelformat=empty, labelsep=none}
 \begin{figure}
    \centering
    \begin{subfigure}[t]{0.35\textwidth}
        \resizebox{5cm}{!}{
            \input{figures/tikz/jtexp1.tikz}
        }
        \subcaption*{\emph{Step~1}: Subtree simulation.}
    \end{subfigure}
~\hspace{7mm}%
    \begin{subfigure}[t]{0.35\textwidth}
        \resizebox{5cm}{!}{
            \input{figures/tikz/jtexp1c.tikz}
        }
        \subcaption{\emph{Step~2}\(^\ast\): Node creation.}
    \end{subfigure}
    \begin{subfigure}[t]{0.35\textwidth}
        \resizebox{5cm}{!}{
            \input{figures/tikz/jtexp3.tikz}
        }
        \subcaption{\emph{Step~3}\(^\ast\): Structure replication.}
    \end{subfigure}
~\hspace{7mm}%
    \begin{subfigure}[t]{0.35\textwidth}
        \resizebox{5cm}{!}{
            \input{figures/tikz/jtexp4.tikz}
        }
        \subcaption{\emph{Step~4}\(^\ast\): Neighbor relocation.}
    \end{subfigure}
    \caption{A possible outcome of the \jteabbr\, where a non-empty subtree was drawn in the expansion of the junction tree in Figure~\ref{fig:graph_and_jt}.} 
    \label{fig:expanded_graph_and_jt_nonempty}
\end{figure} 
\end{myex} 

%% file: figures/tikz/jtexp1_graph.tikz

\tikzstyle{c}=[circle,
                                    thick,
                                    minimum size=0.9cm,
                                    draw=black!100
                                    ]

\begin{tikzpicture}[>=latex,text height=1.5ex,text depth=0.25ex]

  \matrix[row sep=0.5cm,column sep=0.5cm] {
        &&&
        \node (1) [c]{$1$}; &
        &
        \node (2) [c]{$2$}; &
        \\
        \node (3) [c]{$3$}; &
        &
        \node (4) [c]{$4$}; &
        &
        \node (5) [c]{$5$}; &
        &
        \node (6) [c]{$6$}; &
        \\
        &
        \node (7) [c]{$7$}; &
        &
        \node (8) [c]{$8$}; &
        &
        \node (9) [c]{$9$};
        \\
\\
    };

    \path[-]
        (1) edge[thick] (4)
        (1) edge[thick] (5)
        (2) edge[thick] (5)
        (2) edge[thick] (6)
        (3) edge[thick] (4)
        (4) edge[thick] (5)
        (4) edge[thick] (8)
        (5) edge[thick] (9)
        (5) edge[thick] (8)
        (5) edge[thick] (6)
        (6) edge[thick] (9)
        ;

\end{tikzpicture}

%% file: figures/tikz/jtexp.tikz
\begin{tikzpicture}[>=latex,shorten >=1pt,auto,node distance=3.2cm,
    thick,main node/.style={circle,scale=0.8,fill=white!15,minimum size=0.8 cm,draw,font=\sffamily\Large}]
    \node[main node] (34) [] {$\{3,4\}$};
    \node[main node] (145) [right of=34] {$\{1,4,5\}$};
    \node[main node] (256) [right of=145] {$\{2,5,6\}$};
    \node[main node] (568) [right of=256] {$\{5,6,9\}$};
    \node[main node] (7) [below of=34] {$\{7\}$};
    \node[main node] (458) [below of=145] {$\{4,5,8\}$};

    \path[every node/.style={font=\sffamily\small}]
    (34) edge node [] {$\{4\}$} (145)
    (458) edge node [] {$\emptyset$} (7)
    (145) edge node [] {$\{5\}$} (256)
    (256) edge node [] {$\{5,6\}$} (568)
    (145) edge node [] {$\{4,5\}$} (458);

\end{tikzpicture}

%% file: figures/tikz/exp_jt_isolated.tikz
\begin{tikzpicture}[>=latex,shorten >=1pt,auto,node distance=3.2cm,
    thick,main node/.style={circle,scale=0.8,fill=white!15,minimum size=0.8 cm,draw,font=\sffamily\Large}]
    \node[main node] (34) [] {$\{3,4\}$};
    \node[main node] (145) [right of=34] {$\{1,4,5\}$};
    \node[main node] (256) [right of=145] {$\{2,5,6\}$};
    \node[main node] (569) [right of=256] {$\{5,6,9\}$};
    \node[main node] (458) [below of=145] {$\{4,5,8\}$};
    \node[main node] (7) [below of=34] {$\{7\}$};
    \node[main node] (10) [below of=256, draw=red] {$\{10\}$};
    \path[every node/.style={font=\sffamily\small}]
    (34) edge node [] {$\{4\}$} (145)
    (458) edge node [] {$\emptyset$} (7)
    (145) edge node [] {$\{5\}$} (256)
    (256) edge node [] {$\{5,6\}$} (569)
    (256) edge node [] {$\emptyset$} (10)
    (145) edge node [] {$\{4,5\}$} (458);
\end{tikzpicture}

%% file: figures/tikz/forest_emptyset.tikz
\begin{tikzpicture}[>=latex,shorten >=1pt,auto,node distance=3.2cm,
    thick,main node/.style={circle,scale=0.8,fill=white!15,minimum size=0.8 cm,draw,font=\sffamily\Large}]
    \node[main node] (34) [] {$\{3,4\}$};
    \node[main node] (145) [right of=34] {$\{1,4,5\}$};
    \node[main node] (256) [right of=145] {$\{2,5,6\}$};
    \node[main node] (569) [right of=256] {$\{5,6,9\}$};
    \node[main node] (458) [below of=145] {$\{4,5,8\}$};
    \node[main node] (7) [below of=34] {$\{7\}$};
    \node[main node] (10) [below of=256, draw=red] {$\{10\}$};
    \path[every node/.style={font=\sffamily\small}]
    (34) edge node [] {$\{4\}$} (145)
    (145) edge node [] {$\{5\}$} (256)
    (256) edge node [] {$\{5,6\}$} (569)
    (145) edge node [] {$\{4,5\}$} (458);
\end{tikzpicture}

%% file: figures/tikz/exp_jt_isolated_1.tikz
\begin{tikzpicture}[>=latex,shorten >=1pt,auto,node distance=3.2cm,
    thick,main node/.style={circle,scale=0.8,fill=white!15,minimum size=0.8 cm,draw,font=\sffamily\Large}]
    \node[main node] (34) [] {$\{3,4\}$};
    \node[main node] (145) [right of=34] {$\{1,4,5\}$};
    \node[main node] (256) [right of=145] {$\{2,5,6\}$};
    \node[main node] (569) [right of=256] {$\{5,6,9\}$};
    \node[main node] (458) [below of=145] {$\{4,5,8\}$};
    \node[main node] (7) [below of=34] {$\{7\}$};
    \node[main node] (10) [below of=256, draw=red] {$\{10\}$};
    \path[every node/.style={font=\sffamily\small}]
    (34) edge node [] {$\{4\}$} (145)
    (34) edge node [] {$\emptyset$} (7)
    (145) edge node [] {$\{5\}$} (256)
    (256) edge node [] {$\{5,6\}$} (569)
    (458) edge node [] {$\emptyset$} (10)
    (145) edge node [] {$\{4,5\}$} (458);
\end{tikzpicture}

%% file: figures/tikz/jtexp1.tikz
\begin{tikzpicture}[>=latex,shorten >=1pt,auto,node distance=3.2cm,
    thick,main node/.style={circle,scale=0.8,fill=white!15,minimum size=0.8 cm,draw,font=\sffamily\Large}]
    \node[main node] (34) [draw=blue] {$\{3,4\}$};
    \node[main node] (145) [draw=blue,right of=34] {$\{1,4,5\}$};
    \node[main node] (256) [draw=blue,right of=145] {$\{2,5,6\}$};
    \node[main node] (569) [right of=256] {$\{5,6,9\}$};
    \node[main node] (458) [below of=145] {$\{4,5,8\}$};
    \node[main node] (7) [below of=34] {$\{7\}$};
    \path[every node/.style={font=\sffamily\small}]
    (34) edge node [] {$\{4\}$} (145)
    (458) edge node [] {$\emptyset$} (7)
    (145) edge node [] {$\{5\}$} (256)
    (256) edge node [] {$\{5,6\}$} (569)
    (145) edge node [] {$\{4,5\}$} (458);
\end{tikzpicture}

%% file: figures/tikz/jtexp1c.tikz
\begin{tikzpicture}[>=latex,shorten >=1pt,auto,node distance=3.2cm,
    thick,main node/.style={circle,scale=0.8,fill=white!15,minimum size=0.8 cm,draw,font=\sffamily\Large}]
    \node[main node] (3410) [draw=red] {$\{3,4,10\}$};
    \node[main node] (4510) [draw=red,right of=3410] {$\{4,5,10\}$};
    \node[main node] (5610) [draw=red,right of=4510] {$\{5,6,10\}$};
    \node[main node] (145) [draw=blue,below of =4510] {$\{1,4,5\}$};
    \node[main node] (256) [draw=blue,right of=145] {$\{2,5,6\}$};
    \node[main node] (568) [right of=256] {$\{5,6,9\}$};
    \node[main node] (457) [below of=145] {$\{4,5,8\}$};
    \node[main node] (7) [left of=457] {$\{7\}$};
    \path[every node/.style={font=\sffamily\small}]
    (457) edge node [] {$\emptyset$} (7)
    (145) edge node [] {$\{5\}$} (256)
    (256) edge node [] {$\{5,6\}$} (568)
    (145) edge node [] {$\{4,5\}$} (457);

\end{tikzpicture}

%% file: figures/tikz/jtexp3.tikz
\begin{tikzpicture}[>=latex,shorten >=1pt,auto,node distance=3.6cm,
    thick,main node/.style={circle,scale=0.8,fill=white!15,minimum size=0.8 cm,draw,font=\sffamily\Large}]
    \node[main node] (3410) [draw=red] {$\{3,4,10\}$};
    \node[main node] (4510) [draw=red,right of=3410] {$\{4,5,10\}$};
    \node[main node] (5610) [draw=red,right of=4510] {$\{5,6,10\}$};

    \node[main node] (145) [draw=blue,below of=4510] {$\{1,4,5\}$};
    \node[main node] (256) [draw=blue,right of=145] {$\{2,5,6\}$};
    \node[main node] (568) [right of=256] {$\{5,6,9\}$};
    \node[main node] (457) [below of=145] {$\{4,5,8\}$};
    \node[main node] (7) [left of=457] {$\{7\}$};
    \path[every node/.style={font=\sffamily\small}]

    (3410) edge node [] {$\{4,10\}$} (4510)
    (457) edge node [] {$\emptyset$} (7)

    (4510) edge node [] {$\{5,10\}$} (5610)
    (4510) edge node [] {$\{4,5\}$} (145)

    (5610) edge node [] {$\{5,6\}$} (256)

    (256) edge node [] {$\{5,6\}$} (568)
    (145) edge node [] {$\{4,5\}$} (457);

\end{tikzpicture}

%% file: figures/tikz/jtexp4.tikz

\begin{tikzpicture}[>=latex,shorten >=1pt,auto,node distance=3.6cm,
    thick,main node/.style={circle,scale=0.8,fill=white!15,minimum size=0.8 cm,draw,font=\sffamily\Large}]
    \node[main node] (3410) [draw=red] {$\{3,4,10\}$};
    \node[main node] (4510) [draw=red,right of=3410] {$\{4,5,10\}$};
    \node[main node] (5610) [draw=red,right of=4510] {$\{5,6,10\}$};

    \node[main node] (145) [draw=blue,below of=4510] {$\{1,4,5\}$};
    \node[main node] (256) [draw=blue,right of=145] {$\{2,5,6\}$};
    \node[main node] (568) [right of=256] {$\{5,6,9\}$};
    \node[main node] (457) [below of=145] {$\{4,5,8\}$};
    \node[main node] (7) [left of=457] {$\{7\}$};
    \path[every node/.style={font=\sffamily\small}]

    (3410) edge node [] {$\{4,10\}$} (4510)
    (457) edge node [] {$\emptyset$} (7)
    (4510) edge node [] {$\{5,10\}$} (5610)
    (4510) edge node [] {$\{4,5\}$} (145)

    (5610) edge node [] {$\{5,6\}$} (256)

    (5610) edge node [] {$\{5,6\}$} (568)
    (145) edge node [] {$\{4,5\}$} (457);

\end{tikzpicture}

%% file: figures/tikz/algrun/algrun_1_jt.tikz

\input{figures/tikz/node_style.tikz}

\begin{tikzpicture}[>=latex,shorten >=1pt,auto,node distance=3.2cm,
    thick,main node/.style={circle,scale=0.4,fill=white!15,minimum size=0.8 cm,draw,font=\sffamily\Large}]
    \node[main node] (1) [draw=black] {$\{1\}$};

\end{tikzpicture}

%% file: figures/tikz/algrun/algrun_1_graph.tikz

\input{figures/tikz/node_style.tikz}
\begin{tikzpicture}[>=latex,text height=0.9ex,text depth=0.15ex]
  \matrix[row sep=0.2cm, column sep=0.2cm] {
        \node (1) [c]{$1$}; \\
    };

    \path[-]
        ;
\end{tikzpicture}

%% file: figures/tikz/algrun/algrun_1_jt_p.tikz

\input{figures/tikz/node_style.tikz}

\begin{tikzpicture}[>=latex,shorten >=1pt,auto,node distance=3.2cm,
    thick,main node/.style={circle,scale=0.4,fill=white!15,minimum size=0.8 cm,draw,font=\sffamily\Large}]
    \node[main node] (1) [draw=blue] {$\{1\}$};

\end{tikzpicture}

%% file: figures/tikz/algrun/algrun_1_graph_p.tikz

\input{figures/tikz/node_style.tikz}

\begin{tikzpicture}[>=latex,text height=0.9ex,text depth=0.15ex]
  \matrix[row sep=0.2cm, column sep=0.2cm] {
        \node (1) [subtree]{$1$}; \\
    };

    \path[-]
        ;
\end{tikzpicture}

%% file: figures/tikz/algrun/algrun_2_jt.tikz

\input{figures/tikz/node_style.tikz}
\begin{tikzpicture}[>=latex,shorten >=1pt,auto,node distance=3.2cm,
    thick,main node/.style={circle,scale=0.4,fill=white!15,minimum size=0.8 cm,draw,font=\sffamily\Large}]
    \node[main node] (12) [draw=red] {$\{1,2\}$};

\end{tikzpicture}



%% file: figures/tikz/algrun/algrun_2_graph.tikz

\input{figures/tikz/node_style.tikz}

\begin{tikzpicture}[>=latex,text height=0.9ex,text depth=0.10ex]
  \matrix[row sep=0.2cm, column sep=0.2cm] {
        \node (1) [c]{$1$}; &  \node (2) [new]{$2$}; \\
    };

    \path[-]
    (1) edge[thick] (2)
        ;
\end{tikzpicture}

%% file: figures/tikz/algrun/algrun_2_jt_p.tikz

\input{figures/tikz/node_style.tikz}

\begin{tikzpicture}[>=latex,shorten >=1pt,auto,node distance=3.2cm,
    thick,main node/.style={circle,scale=0.4,fill=white!15,minimum size=0.8 cm,draw,font=\sffamily\Large}]
    \node[main node] (1) [draw=black] {$\{1,2\}$};

\end{tikzpicture}

%% file: figures/tikz/algrun/algrun_2_graph_p.tikz

\input{figures/tikz/node_style.tikz}

\begin{tikzpicture}[>=latex,text height=0.9ex,text depth=0.10ex]
  \matrix[row sep=0.2cm, column sep=0.2cm] {
        \node (1) [c]{$1$}; & \node (2) [c]{$2$}; \\
    };

    \path[-]
    (1) edge[thick] (2)
        ;
\end{tikzpicture}

%% file: figures/tikz/algrun/algrun_3_jt.tikz
\begin{tikzpicture}[>=latex,shorten >=1pt,auto,node distance=3.2cm,
    thick,main node/.style={circle,scale=0.4,fill=white!15,minimum size=0.8 cm,draw,font=\sffamily\Large}]
    \node[main node] (12) [draw=black] {$\{1,2\}$};
    \node[main node] (3) [draw=red,right of=12] {$\{3\}$};
    \path[every node/.style={font=\sffamily\small}]
    (12) edge node [] {$\emptyset$} (3);

\end{tikzpicture}

%% file: figures/tikz/algrun/algrun_3_graph.tikz

\input{figures/tikz/node_style.tikz}

\begin{tikzpicture}[>=latex,text height=0.9ex,text depth=0.10ex]
  \matrix[row sep=0.2cm, column sep=0.2cm] {
        \node (1) [c]{$1$}; & \node (2) [c]{$2$}; & \node (3) [new]{$3$}; \\
    };

    \path[-]
    (1) edge[thick] (2)
        ;
\end{tikzpicture}

%% file: figures/tikz/algrun/algrun_3_jt_p.tikz
\begin{tikzpicture}[>=latex,shorten >=1pt,auto,node distance=3.2cm,
    thick,main node/.style={circle,scale=0.4,fill=white!15,minimum size=0.8 cm,draw,font=\sffamily\Large}]
    \node[main node] (12) [draw=blue] {$\{1,2\}$};
    \node[main node] (3) [draw=blue,right of=12] {$\{3\}$};
    \path[every node/.style={font=\sffamily\small}]
    (12) edge node [] {$\emptyset$} (3);
\end{tikzpicture}

%% file: figures/tikz/algrun/algrun_3_graph_p.tikz

\input{figures/tikz/node_style.tikz}

\begin{tikzpicture}[>=latex,text height=0.9ex,text depth=0.10ex]
  \matrix[row sep=0.2cm, column sep=0.2cm] {
        \node (1) [c]{$1$}; & \node (2) [subtree]{$2$}; & \node (3) [subtree]{$3$}; \\
    };

    \path[-]
    (1) edge[thick] (2)
        ;
\end{tikzpicture}

%% file: figures/tikz/algrun/algrun_4_jt.tikz
\begin{tikzpicture}[>=latex,shorten >=1pt,auto,node distance=3.2cm,
    thick,main node/.style={circle,scale=0.4,fill=white!15,minimum size=0.8 cm,draw,font=\sffamily\Large}]
    \node[main node] (12) [draw=black] {$\{1,2\}$};
    \node[main node] (24) [draw=red, right of=12] {$\{2,4\}$};
    \node[main node] (34) [draw=red, right of=24] {$\{3,4\}$};
    \path[every node/.style={font=\sffamily\small}]
    (12) edge node [] {$\{2\}$} (24)
    (24) edge node [] {$\{4\}$} (34);
\end{tikzpicture}

%% file: figures/tikz/algrun/algrun_4_graph.tikz

\input{figures/tikz/node_style.tikz}

\begin{tikzpicture}[>=latex,text height=0.9ex,text depth=0.10ex]
  \matrix[row sep=0.2cm, column sep=0.2cm] {
        \node (1) [c]{$1$}; & \node (2) [c]{$2$}; & \node (4) [new]{$4$}; & \node (3) [c]{$3$}; \\
    };

    \path[-]
    (1) edge[thick] (2)
    (2) edge[thick] (4)
    (3) edge[thick] (4)
        ;
\end{tikzpicture}

%% file: figures/tikz/algrun/algrun_4_jt_p.tikz
\begin{tikzpicture}[>=latex,shorten >=1pt,auto,node distance=3.2cm,
    thick,main node/.style={circle,scale=0.4,fill=white!15,minimum size=0.8 cm,draw,font=\sffamily\Large}]
    \node[main node] (12) [draw=black] {$\{1,2\}$};
    \node[main node] (24) [draw=blue, right of=12] {$\{2,4\}$};
    \node[main node] (34) [draw=black, right of=24] {$\{3,4\}$};
    \path[every node/.style={font=\sffamily\small}]
    (12) edge node [] {$\{2\}$} (24)
    (24) edge node [] {$\{4\}$} (34);
\end{tikzpicture}

%% file: figures/tikz/algrun/algrun_4_graph_p.tikz

\input{figures/tikz/node_style.tikz}

\begin{tikzpicture}[>=latex,text height=0.9ex,text depth=0.10ex]
  \matrix[row sep=0.2cm, column sep=0.2cm] {
        \node (1) [c]{$1$}; & \node (2) [subtree]{$2$}; & \node (4) [subtree]{$4$}; & \node (3) [c]{$3$}; \\
    };

    \path[-]
    (1) edge[thick] (2)
    (2) edge[thick] (4)
    (3) edge[thick] (4)
        ;
\end{tikzpicture}

%% file: figures/tikz/algrun/algrun_5_jt.tikz
\begin{tikzpicture}[>=latex,shorten >=1pt,auto,node distance=3.2cm,
    thick,main node/.style={circle,scale=0.4,fill=white!15,minimum size=0.8 cm,draw,font=\sffamily\Large}]
    \node[main node] (12) [draw=black] {$\{1,2\}$};
    \node[main node] (245) [draw=red, right of=12] {$\{2,4,5\}$};
    \node[main node] (34) [draw=black, right of=245] {$\{3,4\}$};
    \path[every node/.style={font=\sffamily\small}]
    (12) edge node [] {$\{2\}$} (245)
    (245) edge node [] {$\{4\}$} (34);
\end{tikzpicture}

%% file: figures/tikz/algrun/algrun_5_graph.tikz

\input{figures/tikz/node_style.tikz}

\begin{tikzpicture}[>=latex,shorten >=1pt,auto,node distance=3.2cm,
    thick,main node/.style={circle,scale=0.4,fill=white!15,minimum size=0.8 cm,draw,font=\sffamily\Large}]
  \matrix[row sep=0.2cm, column sep=0.2cm] {
        \node (1) [c]{$1$}; & \node (2) [c]{$2$}; & \node (4) [c]{$4$}; & \node (3) [c]{$3$}; \\
        && \node (5) [new]{$5$}; &\\
    };

    \path[-]
    (1) edge[thick] (2)
    (2) edge[thick] (4)
    (2) edge[thick] (5)
    (3) edge[thick] (4)
    (4) edge[thick] (5)
    ;
\end{tikzpicture}

%% file: ex_cta_traj.tex
\begin{myex}
\label{ex:alg_run}
\autoref{fig:alg_run} should be read in chunks of two rows (except for the first row) and shows the junction trees, the corresponding decomposable graphs and the subgraphs generated by the \jteabbr\, for $\jtind\in\{1,\dots,5\}.$
The left column shows the expansion of the junction trees and the right column shows the underlying decomposable graphs. Subtrees are colored in blue and the new nodes are colored in red. Unaffected nodes are black. Vertices in the underlying graphs are colored analogously.
For example, the subtree $\tree_2^\prime$ selected in the generation of $\tree_3$ on Row~5 is found on Row~4. 
The underlying nodes in $\tree_2^\prime$ for creating $\tree_3$ is also found on Row~4, and so on. 
Note that the subtree $\tree_2'$ used in the creation of $\tree_3$, is the empty tree, thus $\tree_2'$ is black.
The tuning parameters of the junction tree expander are set to $\subtreecont=0.3$ and $\subtreenonempty=0.9$.
\end{myex}

%% file: rcta.tex
In this section, we present the \emph{\jtc}, a reversed version of the \jteabbr, introduced in the previous section.
The idea is to collapse a junction tree $\tree[+]\in \trsp[\jtind+1]$ into a new tree $\tree \in \trsp[\jtind]$ by removing  $\gennodea[\jtind+1]$ from the underlying graph in such a way that $\tree \in \supp(\prop{\jtind}(\cdot,\tree[+]))$.
As will be proved in this section, this procedure defines a Markov kernel $\bk{\jtind} : \trsp[\jtind+1] \times \potset(\trsp[\jtind]) \rightarrow [0,1]$. 

Next follows a description of the different suboperations in the sampling procedure for $\bk{\jtind}$.
The details of the steps are given in Algorithm~\ref{alg:rcta} in Appendix~\ref{appendix:a}.

\begin{jtcolalg}
Let $\tree[+]$ be a junction tree in $\trsp[\jtind+1]$.
Similarly to the \jteabbr, the \jtcabbr\, takes two different forms depending on whether ${\{\gennodea[\jtind+1]\}}$ is present as a node in $\tree[+]$ or not.

If \(\{\gennodea[\jtind+1]\}\) is a node in \(\tree[+]\) proceed as follows:
\setlist[enumerate]{align=left}
\begin{enumerate}
 \setlength\itemsep{1em}
    \item[Step~1. ] Remove \(\{\gennodea[\jtind+1]\}\) and it incident links to obtain a forest, possibly containing only one tree.
    \item[Step~2. ] Randomly connect the forest into a tree. 
\end{enumerate}

If \(\{\gennodea[\jtind+1]\}\) is not a node in \(\tree[+]\) proceed as follows:
\begin{enumerate}
 \setlength\itemsep{1em}
    \item[Step~1\(^\ast\). ] Let \(\tree[+]'=(\graphnodeset[+]',\graphedgeset[+]')\) be the subtree of \(\tree[+]\) induced by the nodes containing the vertex \(\gennodea[\jtind+1]\) and  enumerate the nodes in \(\graphnodeset[+]'\) by \((\jtnewnode{\jtnodeind})_{\jtnodeind=1}^{|\graphnodeset[+]'|} \).
    \item[Step~2\(^\ast\). ]
   For all \(\jtnodeind \in \{1, \ldots, |\graphnodeset[+]'| \}\), draw at random \(\jtnode_{\jtnodeind}\) from \(\jtpotorigins{\jtnodeind}\), the set of neighbors of \(\jtnewnode{\jtnodeind}\) in \(\tree[+]\) having the associated separator 
    \(\jtnewnode{\jtnodeind}\setminus \{\gennodea[\jtind+1]\}\). 
    If no such neighbor exists, let \(\jtnode_{\jtnodeind} = \jtnewnode{\jtnodeind}\setminus \{\gennodea[\jtind+1]\}\).
    \item[Step~3\(^\ast\). ] Replace each node \(\jtnewnode{\jtnodeind}\) by the corresponding node \(\jtnode_{\jtnodeind}\) in the sense that \(\jtnode_{\jtnodeind}\) is assigned all former neighbors of \(\jtnewnode{\jtnodeind}\).
\end{enumerate}
\end{jtcolalg}
The next example illustrates a reversed version of Example~\ref{ex:cta_expansion}.
\begin{myex}\label{ex:jtc}
Consider collapsing the junction tree in  the bottom right panel of Figure~\ref{fig:expanded_graph_and_jt_nonempty} by the vertex 10.
The induced subgraph \(\tree[+]\), having the nodes \(\jtnewnode{1}=\{3,4,10\}, \jtnewnode{2}=\{4,5,10\}\), and \(\jtnewnode{3}=\{5,6,10\}\) is colored in red in the same subfigure.
Further we see that 
\(\jtpotorigins{1}=\emptyset\) implies that \(\jtnode_1=\{3,4\}\) and  \(\jtpotorigins{2}=\{\{1,4,5\}\}\) implies  \(\jtnode_{2} = \{1,4,5\}\).
By drawing \(\jtnode_{3}=\{2,5,6\}\) from \(\jtpotorigins{3}=\{\{2,5,6\},  \{5,6,9\}\}\), the junction tree in the top left panel of Figure~\ref{fig:expanded_graph_and_jt_nonempty} is obtained.
\end{myex}

The induced transition probability of collapsing $\tree[+] \in \trsp[\jtind+1]$ into a tree $\tree \in \supp(\bk{\jtind}(\tree[+], \cdot))$ has the form
\begin{align*}
    \bk{\jtind}(\tree[+], \tree) =
    \begin{cases}
        1/\jtsepcount_\mathsmaller{\trgr(\tree)}(\emptyset) &\text{ if } \{\gennodea[\jtind+1]\} \in \graphnodeset[+],\\
        1/\prod_{\jtnodeind=1}^{|\graphnodeset[+]'|}\max(1, |\jtpotorigins{\jtnodeind}|) &\text{ otherwise,}
    \end{cases}
\end{align*}
where, as before, $\graphnodeset[+]'$ is the set of nodes in $\tree[+]$ containing  $\gennodea[\jtind+1]$. 
The max operation is needed in order to make the expression well defined even when $\jtpotorigins{\jtnodeind}$ is empty.

The \jtcabbr\, is a reversed version of the \jteabbr\, in the sense that for all \(\jtind \in \nset\), a junction tree \(\tree\in \trsp[\jtind]\), generated by the \jtcabbr\, from a junction tree \(\tree[+]\in \trsp[\jtind+1]\), can be used as input to the \jteabbr\, to generate \(\tree[+]\).
This property is formulated in the next theorem.
\begin{mythm}
    \label{thm:rcta}
    For all $m \in \nset$ and $\tree[+]\in \trsp[m+1]$,
    \begin{enumerate}[label=(\roman*)] 
    \setlength\itemsep{0.5em}
        \item 
            $\supp(\bk{m}(\tree[+], \cdot)) \subseteq \trsp[m]$, \label{itm:rcta:jt_preservation_prop}
        \item 
            $\supp(\bk{m}(\tree[+], \cdot)) \subseteq \supp(\prop{m}(\cdot,\tree[+]))$,  \label{itm:rcta:valid_cta_expansion}
        \item 
            $\trgr(\tree) = \trgr(\tree[+])[\{\gennodea[\ell]\}_{\ell=1}^{\jtind}]$ for any $\tree \in \supp(\bk{m}(\tree[+], \cdot))$. \label{itm:rcta:jt_ind_subgraph_prop}
    \end{enumerate}
\end{mythm}

Theorem~\ref{thm:rcta} proves to be crucial in the  \smcabbr\, context described in Section~\ref{sec:smc} and in particular in the refreshment step of the particle Gibbs sampler detailed in \citep{olsson2019}.

%% file: update_mu.tex
\citet{doi:10.1198/jcgs.2009.07129} provide an expression for counting the number of equivalent junction trees of a given decomposable graph.
In this section we derive a factorisation of the same expression which shows to alleviate the computational burden when calculated for expanded graphs.
For sake of completeness, we restate three theorems from \citep{doi:10.1198/jcgs.2009.07129}.
The first counts the number of ways a forest can be reconnected into a tree and was first established in \citep{moon1967enumerating}.
\begin{mythm}[\citet{moon1967enumerating}]
	\label{thm:num_forest}
	The number of distinct ways that a forest 
	of order 
	$\jtind$ 
	comprising $q$ subtrees of 
	orders 
	$r_1,\dots,r_q$ can be connected into a single tree by adding $q-1$ edges is
	\begin{align*}
    	\jtind^{q-2}\prod_{i=1}^qr_i.
	\end{align*}
\end{mythm}
For a given junction tree $\tree$, let $t_{\sep}$ denote the order of the subtree $\tree_\sep$ induced by the separator $\sep$. 
Now, let $m_\sep$ be the number of links associated with $\sep$ and let $f_1,\dots,f_{m_\sep+1}$ be the orders of the 
tree components 
in $\jtsepforest{\sep}(\tree)$.
Then, by Theorem \ref{thm:num_forest} the following is obtained.
\begin{mythm}[\citet{doi:10.1198/jcgs.2009.07129}]
\label{thm:nu_g}
	The number of ways that the components of $\jtsepforest{\sep}(\tree)$, where $\sep$ is a separator in a graph $\graph$ with junction tree $\tree$, can be connected into a single tree by adding the appropriate number of links is given by
	\begin{align*}
	    \jtsepcount_\mathsmaller{\graph}(\sep)=t_\sep^{m_\sep-1}\prod_{j=1}^{m_\sep+1}f_j.
	\end{align*}
\end{mythm}
\begin{mythm}[\citet{doi:10.1198/jcgs.2009.07129}]
\label{thm:mu_g}
	The number of junction trees for a decomposable graph $\graph$ is given by
	\begin{align*}
	    \jtcount{\graph} = \prod\limits_{\sep \in \seps{\graph}}\jtsepcount_\mathsmaller{\graph}(\sep).
	\end{align*}
\end{mythm}

In the sequential sampling context considered in this paper it is useful to exploit that any decomposable graph $\graph[+]\in \graphsp[\jtind+1]$ can be regarded as an expansion of another decomposable graph $\graph\in \graphsp[\jtind]$, in the sense that $\graph[+]$ is obtained by expanding $\graph$ with the vertex $\gennodea[\jtind+1]$.
This follows for example by induction using \citep[Corollary 2.8]{lauritzen1996}. 



The key insight when calculating \(\mu(\graph[+])\) is that when a vertex is added to $\graph$, not all separators will necessarily be affected. 
This implies that $\jtsepcount_{\graph}(\sep) = \jtsepcount_{\graph[+]}(\sep)$ for some separators. 

\begin{mythm}
\label{thm:update_mu}
	Let $\graph[+]\in \graphsp[\jtind+1]$ be an expansion  
	of some graph $\graph\in \graphsp[\jtind]$ by the extra vertex $\gennodea[\jtind+1]$.
	Let $\sepset^\star \subseteq \seps{\graph[+]}$ be the set of unique separators created (note that \(\sepset^\star\cap \seps{\graph}\) might be non-empty) by the expansion.
	Then 
	\begin{align}
	\label{eq:mufac}
		\jtcount{\graph[+]} = \frac{\prod_{\sep \in {\newseps{\graph[+]}}}{\jtsepcount_\mathsmaller{\graph[+]}(\sep)}}{ \prod_{\sep \in {\newseps{\graph}}}\jtsepcount_\mathsmaller{\graph}(\sep) }\jtcount{\graph},
	\end{align}
	 where $\newseps{\graph} \eqdef \{\sep\in \seps{\graph} : \exists \sep^\prime \in \sepset^\star  \text{ such that } \sep\subseteq \sep^\prime\}$
	is the set of separators in $\graph$ contained in some separator in $\sepset^\star$.
\end{mythm}
The potential computational gain obtained by using the factorisation in Theorem~\ref{thm:update_mu} 
is illustrated by the following example.

\begin{myex}
\label{ex:updatemu}
Let $\graph[+]\in \graphsp[\jtind+1]$ be an expansion of a graph $\graph\in \graphsp[\jtind]$ in the sense that $\gennodea[\jtind+1]$ is connected to every vertex in one of the cliques in $\graph$.
Then, since the set of separators is the same in the two graphs, it holds that $\jtcount{\graph[+]}=\jtcount{\graph}.$
\end{myex}

%% file: SMC-new.tex
We demonstrate how the \jteabbr\, and the \jtcabbr\, can be cast into the framework of \smcabbr\, methods (see  \cite{chopin:papaspiliopoulos:2020} for a recent introduction) for sampling from a sequence $(\targ{m})_{m \in \nsetpos}$ of probability distributions, where each $\targ{m}$ is a distribution on \(\trsp[m]\). For every $m$ we assume that $\targ{m}$ is known only up to a normalising constant, i.e., $\targ{m} \propto \untarg{m}$, where $\untarg{m}$ is a tractable, unnormalised function. Following \citep{10.2307/3879283}, we introduce path spaces \(\trsp[1:\jtind]\eqdef\trsp[1]\times \dots \times \trsp[\jtind] \) and
let 
\begin{align} \label{eq:norm_embeded}
\targpath{m}(\tree_{1:m}) &
\eqdef \prod_{\ell=1}^{m-1}\frac{\targ{\ell+1}(\tree_{\ell+1})}{\targ{\ell}(\tree_{\ell})} \bk{\ell}(\tree_{\ell+1},\tree_{\ell}), \quad \tree_{1:\jtind} \in \trsp[1:\jtind],
\end{align}
be extended target distributions. Importantly, each target \(\targ{m}\) is the marginal of \(\targpath{m}\) with respect to the \(m\)th component. In many applications, the aim is to sample from a given distribution $\pi$ on some junction-tree space $\trsp[n]$ induced by $n$ vertices, and in this case one may let $\targ{n} = \pi$ and $(\targ{\ell})_{m = 1}^{n - 1}$ be the marginals of $\pi$ (if these are known up to normalising constants), serving to guide the distribution flow towards the target $\pi$. 

Now, introduce, for all $m$, proposal distributions 
\begin{align} \label{eq:proppath}
\bar \rho_\jtind(\tree_{1:\jtind})& 
\eqdef 
\prod_{\ell=1}^{\jtind-1}\prop{\ell}(\tree_\ell, \tree_{\ell+1}), \quad \tree_{1:\jtind} \in \trsp[1:\jtind].  
\end{align}
Since Theorem~\ref{thm:rcta} implies that \(\supp(\bk{\ell}(\cdot,\tree_{\ell})) \subseteq \supp (\prop{\ell}(\tree_\ell, \cdot))\) for all \(\ell \in \{1,\dots, \jtind - 1\}\), it is readily checked that \( \supp(\targpath{m}) \subseteq \supp(\bar \rho_\jtind)\).
This property, along with Theorems~\ref{thm:jt_preservation} and \ref{thm:characterization}, allows the extended target distributions \eqref{eq:norm_embeded} to be sampled by means of an importance-sampling procedure, where independent tree paths $\treepart{1:m}{i} = (\treepart{1}{i},\ldots,\treepart{m}{i})$, $i \in \{1, \ldots, N\}$, generated sequentially using the \jteabbr, are assigned importance weights 
\begin{equation} \label{eq:SIS:weights}
\wgt{m}{i} \eqdef \prod_{\ell = 1}^{m - 1} \frac{ \untarg{\ell+1}(\treepart{\ell+1}{i}) \bk{\ell}(\treepart{\ell + 1}{i},\treepart{\ell}{i})}{ \untarg{\ell}(\treepart{\ell}{i})\prop{\ell}(\treepart{\ell}{i} , \treepart{\ell+1}{i})}
             \frac{  }{ } \propto \frac{\targpath{m}(\treepart{1:m}{i})}{\bar \rho_\jtind(\treepart{1:m}{i})}. 
\end{equation}
Here $N$ is the Monte Carlo sample size. Thanks to the Markovian structure of the proposal \eqref{eq:proppath} and the multiplicative structure of the weights \eqref{eq:SIS:weights}, this procedure can be implemented sequentially by applying recursively the update described in Algorithm~\ref{alg:SIS:update}. This yields a sequence $(\treepart{m}{i}, \wgt{m}{i})_{i = 1}^N$, $m \in \nsetpos$, of weighted samples,   where, since $\targ{m}$ is the marginal of $\targpath{m}$ with respect to the last component, $\sum_{i = 1}^N \wgt{m}{i} h(\treepart{m}{i}) / \wgtsum{m}$, with $\wgtsum{m} \eqdef \sum_{i = 1}^N \wgt{m}{i}$, is a strongly consistent self-normalised estimator of the expectation $\targ{m}(h) \eqdef \sum_{T \in \trsp[m]} h(T) \targ{m}(T)$ of any real-valued test function $h$ under $\targ{m}$. In the SMC literature, the draws $(\treepart{m}{i})_{i = 1}^N$ are typically referred to as \emph{particles}. 

\begin{algorithm}[ht!]
    \KwIn{$(\treepart{m}{i}, \wgt{m}{i})_{i = 1}^\N$}
    \KwOut{$(\treepart{m+1}{i}, \wgt{m+1}{i})_{i = 1}^\N$}
         \For{$i \gets 1, \ldots, \N$}{
             draw \(\treepart{m + 1}{i} \sim  \prop{m}(\treepart{m}{i}, \cdot)\)\; 
              set $\displaystyle \wgt{m + 1}{i} \gets
             \frac{ \untarg{m+1}(\treepart{m+1}{i}) \bk{m}(\treepart{m + 1}{i},\treepart{m}{i})}{ \untarg{m}(\treepart{m}{i})\prop{m}(\treepart{m}{i}, \treepart{m+1}{i})}
             \frac{  }{ } \wgt{m}{i}$\;
         }
     \Return{$(\treepart{m+1}{i}, \wgt{m+1}{i})_{i = 1}^\N$}
     \vspace{0.6cm}
     \caption{Sequential importance sampling.}
     \label{alg:SIS:update}
\end{algorithm}

Even though this \emph{sequential importance sampling} procedure, which is described in Algorithm~\ref{alg:SIS:update}, appears appealing at a first sight, the multiplicative weight updating formula on Line~3 is problematic in the sense that it will, inevitably, lead to severe weight skewness and, consequently, high Monte Carlo variance. In fact, it can be shown that updating the weights in this naive manner leads to a Monte Carlo variance that increases geometrically fast with $m$; see e.g. \cite[Chapter~7.3]{cappe:moulines:ryden:2005} for a discussion. Needless to say, this is impractical for most applications

In order to cope with the weight-degeneracy problem, \citet{gordon1993novel} proposed furnishing the previous sequential importance sampling algorithm with a \emph{selection step}, in which the particles are resampled, with replacement, in proportion to their importance weights. Upon selection, all particles are assigned the unit weight, and  the particles and importance weights are then updated as in Algorithm~\ref{alg:SIS:update}. Such selection is a key ingredient in SMC methods, and it can be shown mathematically that the resulting \emph{sequential importance sampling with resampling} algorithm, which is given in Algorithm~\ref{alg:SMC:update}, is indeed numerically stable \cite{delmoral:guionnet:2001,delmoral:2004}.  
\begin{algorithm}[ht!]
    \KwIn{$(\treepart{m}{i}, \wgt{m}{i})_{i = 1}^\N$}
    \KwOut{$(\treepart{m+1}{i}, \wgt{m+1}{i})_{i = 1}^\N$}
         select $(\treeparttd{m}{i})_{i = 1}^N$ among $(\treepart{m}{i})_{i = 1}^N$ in proportion to  $(\wgt{m}{i})_{i = 1}^N$\; 
         \For{$i \gets 1, \ldots, \N$}{
             draw \(\treepart{m + 1}{i} \sim  \prop{m}(\treeparttd{m}{i}, \cdot)\)\;\label{smc:drawtree}
             set $\displaystyle \wgt{m + 1}{i} \gets
             \frac{ \untarg{m+1}(\treepart{m+1}{i}) \bk{m}(\treepart{m + 1}{i},\treeparttd{m}{i})}{ \untarg{m}(\treeparttd{m}{i})\prop{m}(\treeparttd{m}{i}, \treepart{m+1}{i})}$\;
             \label{alg:smc:weigthupdt}
         }
     \Return{$(\treepart{m+1}{i}, \wgt{m+1}{i})_{i = 1}^\N$}
     \vspace{0.6cm}
     \caption{Sequential importance sampling with resampling.}
     \label{alg:SMC:update}
\end{algorithm}

In standard self-normalised importance sampling, the average weight provides an unbiased estimator of the normalising constant of the target. However, when the particles are resampled systematically, as in Algorithm~\ref{alg:SMC:update}, this simple estimator is no longer valid. Instead, it is possible to show that for every $m$, the estimator 
$$
\gamma^N_m(h) \eqdef \frac{1}{\N^\jtind}  \left( \prod_{\ell = 1}^{\jtind-1} \wgtsum{\ell} \right)
 \sum_{i = 1}^\N \wgt{\jtind}{i} h(\treepart{\jtind}{i})
$$
is an unbiased estimator of $\gamma_m(h)$ for any real-valued test function $h$. In particular, 
\begin{equation} \label{eq:constest}
\untarg{m}^N(\1_{\trsp[m]}) = \frac{1}{\N^\jtind} \prod_{\ell = 1}^{\jtind} \wgtsum{\ell} 
\end{equation}
provides an unbiased estimator of the normalising constant $\untarg{m}(\1_{\trsp[m]})$ of $\targ{m}$. This estimator will be illustrated in the next section. 

%% file: simulations-new.tex
We demonstrate two applications of Algorithm~\ref{alg:SMC:update} for estimating the cardinalities \(|\graphsp[\jtind]|\) and \(|\trsp[\jtind]|\) of the spaces of decomposable graphs and junction trees, respectively.

\subsection{Estimating  $|\graphsp[\jtind]|$ } \label{sec:estdec}

\citet{countchordal} provides an exact expression for $|\graphsp[\jtind]|$ and evaluates the same for $\jtind \le 13$. In the same reference, the author also establishes the asymptotic expression \(|\graphsp[\jtind]| \sim \sum_{\ell=1}^{\jtind}{\jtind \choose \ell} 2^{\ell(\jtind-\ell)}\).
Another exact algorithm that calculates $|\graphsp[\jtind]|$ for \(\jtind\le10\) is proposed in \citep{kawahara2018enumerating}.



In this study we will use  Algorithm~\ref{alg:SMC:update} for estimating $|\graphsp[m]|$, $m \in \nsetpos$, on the basis of the target probability distributions
$$
\targ{m}(T_m) \propto \untarg{m}(T_m) = \frac{1}{\mu(\trgr(\tree_m))} \1_{\trsp[m]}(T_m). 
$$
Note that the normalising constant $\untarg{m}(\1_{\trsp[m]})$ of $\targ{m}$ equals $|\graphsp[m]|$.; indeed, 
\begin{multline*}
\untarg{m}(\1_{\trsp[m]}) = \sum_{G_m \in \graphsp[m]} \sum_{T_m : \trgr(T_m) = G_m} \frac{1}{\mu(\trgr(\tree_m))} \1_{\trsp[m]}(T_m) \\
= \sum_{G_m \in \graphsp[m]} \frac{1}{\mu(G_m)} \sum_{T_m : \trgr(T_m) = G_m} \1_{\trsp[m]}(T_m) = |\graphsp[m]|.  
\end{multline*}

With this formulation, unbiased estimates 
of $|\graphsp[m]|$, $m \in \nsetpos$, can be obtained directly using \eqref{eq:constest}. Note that in this setting Line~4 of Algorithm~\ref{alg:SMC:update} reduces to
\begin{align}
    \frac{ \untarg{m+1}(\tree_{\jtind+1})}{\untarg{\jtind}(\tree_{m})} = \frac{\mu(\trgr(\tree_{m}))}{\mu(\trgr(\tree_{m+1}))},
\end{align}
where \(\tree_{\jtind}\in \trsp[\jtind]\), \(\tree_{\jtind+1}\in \trsp[\jtind+1]\),
for which, as demonstrated by Example~\ref{ex:updatemu}, the computational burden can be substantially reduced using the factorisation \eqref{eq:mufac} since  \(\tree_{\jtind+1}\in \supp (\prop{\jtind}(\tree_\jtind, \cdot))\).

Table~\ref{tab:numpropdec} shows means and standard errors based on $10$ estimates \(\widehat{|\graphsp[\jtind]|}\) of $|\graphsp[m]|$ for $\jtind \in \nn[15]$.
The upper panel of the table shows  $\widehat{|\graphsp[\jtind]|}$ while the lower panel shows $\widehat{|\graphsp[\jtind]|}/2^{{\jtind(\jtind-1)/ 2}}$, i.e. estimates of the fraction of undirected graphs that are decomposable.
For $\jtind\le 13$ the exact enumerations are given in the second column. 
We ran the  \smcabbr\, sampler with tuning parameters  $\subtreecont=0.5$, $\subtreenonempty=0.5$ and the number of particles was set to $\N=10000$.  
Figure~\ref{fig:asymtpotics} displays the asymptotic behavior of $|\graphsp[\jtind]|$ and \( \widehat{|\graphsp[\jtind]|}\) for $\jtind\le50$, along with the exact values for $\jtind \leq 13$, justifying a concordance with the exact results.
Each of the 10 estimates took about 10 minutes to calculate.
\begin{table}[ht!]
\centering
\begin{tabular}{lllll}
$\p $ &\(|\graphsp[\jtind]|\)  &\(\widehat{|\graphsp[\jtind]|}\) & std. err. \\
\hline

4       &       61        &       60.4        &       0.40& \\

5       &       822        &       815        &       16.2 & \\

6       &       1.82        &       1.78        &       0.37 &$\times  10^4 $\\

7       &       6.18        &       5.92        &       0.15 & $\times 10^5$\\

8       &       3.09        &       3.00        &       0.13 & $\times 10^7$\\

9       &       2.19        &       2.14        &       0.09 & $\times 10^9$\\

10      &       2.15        &       2.11        &       0.10 &$\times  10^{11}$\\

11      &       2.88        &       2.80        &       0.18 &$\times  10^{13}$\\

12      &       5.17        &       5.15        &       0.51 & $\times 10^{15}$\\

13      &       1.23        &       1.21        &       0.17 & $\times 10^{18}$\\

14      &       -        &       3.74        &          0.62 &$\times  10^{20}$ \\

15      &       -        &       1.53        &          0.30 & $\times 10^{23}$\\

&&&\\
$\p $ & $|\graphsp[\p]|/2^{{\p(\p-1)/ 2}}$  &$\widehat{|\graphsp[\p]|}/2^{{\p(\p-1)/ 2}}$ & std. err. & \\
\hline

4       &       9.53        &       9.44        &       0.06 & $\times 10^{-1}$\\

5       &       8.03        &       7.96        &       0.16 & $\times 10^{-1}$\\

6       &       5.54        &       5.43        &       0.11 & $\times 10^{-1}$\\

7       &       2.95        &       2.82        &       0.07 & $\times 10^{-1}$\\

8       &       1.15        &       1.12        &       0.05 & $\times 10^{-1}$\\

9       &       3.19        &       3.11        &       0.13 & $\times 10^{-2}$\\

10      &       6.12        &       5.99        &       0.28 & $\times 10^{-3}$\\

11      &       7.99        &       7.78        &       0.51 & $\times 10^{-4}$\\

12      &       7.00        &       6.98        &       0.70 & $\times 10^{-5}$\\

13      &       4.09        &       3.99        &       0.55 & $\times 10^{-6}$\\

14      &       -        &       1.51        &       0.25 & $\times 10^{-7}$\\

15      &       -        &       3.77        &       0.74 & $\times 10^{-9}$\\

 \end{tabular}
\vspace{0.6cm}

\caption{Sequential Monte Carlo estimation of the number decomposable graphs and the fraction of graphs which are decomposable.}
\label{tab:numpropdec}
\end{table}

\subsection{Estimating  $|\trsp[\jtind]|$ }

As far as we know there is no method available in the literature for efficiently calculating  \(|\trsp[\jtind]|\eqdef\sum_{\graph\in \graphsp[\jtind]}\mu(\graph)\).
However, for \(\jtind \le 5\) it is computationally tractable to first find all the 822 graphs \(\graphsp[\jtind]\) by Monte Carlo sampling and then evaluate \(\mu\) for each of them.

As in Section~\ref{sec:estdec} we find an unbiased estimator of \(|\trsp[\jtind]|\) by constructing target distributions
$$
\targ{m}(\tree_{m}) \propto \untarg{m}(\tree_m) = \1_{\trsp[m]}(\tree_m), 
$$
to that the normalising constant $\untarg{m}(\1_{\trsp[m]})$ equals $|\trsp[m]|$, 
and then use \eqref{eq:constest}.
Note that with this setting the first factor in Line~4 in Algorithm~\ref{alg:SMC:update} simplifies as
\begin{align}
    \frac{ \untarg{m+1}(\tree_{\jtind+1})}{\untarg{\jtind}(\tree_{m})} = 1, 
\end{align}
for all \(\tree_{\jtind}\in \trsp[\jtind]\), \(\tree_{\jtind+1}\in \supp (\prop{\jtind}(\tree_\jtind, \cdot))\).

The third and fourth columns of the upper panel in Table~\ref{tab:numtrees} show estimated means and standard deviations of \(\widehat{|\trsp[\jtind]|}\) for \(\jtind\le 15 \) based on 10 replicates. 
The true values for are shown in the first column for \(\jtind\le 5\).

The lower panel of Table~\ref{tab:numtrees} displays estimates of the number of junction trees per decomposable graph, \(|\trsp[\p]|/|\graphsp[\p]|\), for different numbers of vertices. True numbers as are shown in the first column, and estimated means and standard deviations of  \(\widehat{|\trsp[\p]|}/\widehat{|\graphsp[\p]|}\) are shown in the third and fourth columns. 
Interestingly, Figure~\ref{fig:treespergraph} indicates an exponential growth rate of the estimated junction trees per decomposable graph for \(\p\le 50\).
Each of the 10 estimates took about 6 minutes to compute.
\begin{table}[ht!]
\centering
\begin{tabular}{lllll}
$\p $ &\(|\trsp[\p]|\)  &\(\widehat{|\trsp[\p]|}\) & std. err. \\
\hline

1       &       1        &       1        &       0.00        &       \\

2       &       2        &       2        &       0.00        &       \\

3       &       10        &       10        &       0.03        &      \\

4       &       108        &       106        &       1.41        &      \\

5       &       2.09        &       2.03        &       0.03        &       $\times 10^{3}$ \\

6       &       -        &       6.10        &       0.10        &       $\times 10^{4}$ \\

7       &       -        &       2.76        &       0.08        &       $\times 10^{6}$ \\

8       &       -        &       1.79        &       0.07        &       $\times 10^{8}$ \\

9       &       -        &       1.63        &       0.11        &       $\times 10^{10}$ \\

10      &       -        &       2.00        &       0.18        &       $\times 10^{12}$ \\

11      &       -        &       3.37        &       0.32        &       $\times 10^{14}$ \\

12      &       -        &       7.45        &       0.88        &       $\times 10^{16}$ \\

13      &       -        &       2.12        &       0.30        &       $\times 10^{19}$ \\

14      &       -        &       7.87        &       1.34        &       $\times 10^{21}$ \\

15      &       -        &       3.88        &       0.80        &       $\times 10^{24}$ \\
&&&\\
$\p $ & $|\trsp[\p]|/|\graphsp[\p]|$  &$\widehat{|\trsp[\p]|}/ \widehat{|\graphsp[\p]|}$ & std. err.  \\
\hline

1       &       1        &       1        &       0.00 \\

2       &       1        &       1        &       0.00 \\

3       &       1.25        &       1.25&       0.00 \\

4       &       1.77        &       1.76        &       0.03 \\

5       &       2.54        &       2.49        &       0.07 \\

6       &       -        &       3.43        &       0.1 \\

7       &       -        &       4.66        &       0.23 \\

8       &       -        &       5.99        &       0.45 \\

9       &       -        &       7.67        &       0.64 \\

10      &       -        &       9.51        &       0.90 \\

11      &       -        &       12.1        &       1.35 \\

12      &       -        &       14.6        &       1.98 \\

13      &       -        &       17.9        &       3.11 \\

14      &       -        &       21.5        &       4.23 \\

15      &       -        &       26.1        &       6.16 \\

 \end{tabular}
\vspace{0.6cm}

\caption{Sequential Monte Carlo estimation of the number junction trees and the expected number of junction trees per decomposable graph.}
\label{tab:numtrees}
\end{table}

\begin{figure}[!ht]
    \centering
    \includegraphics[width=0.6\textwidth]{{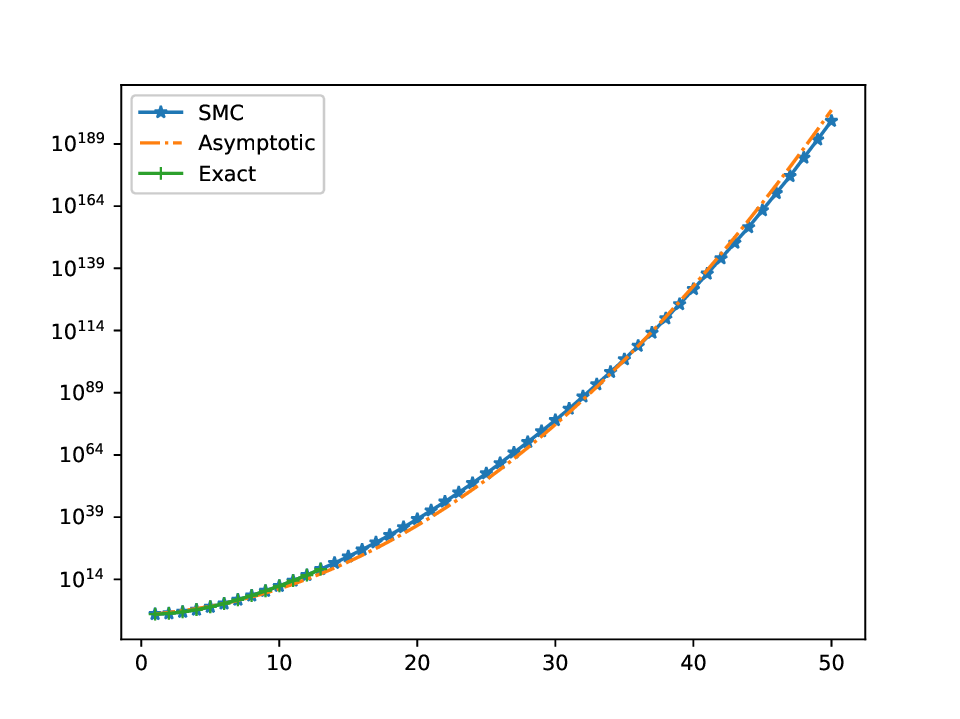}}
    \caption{The number of decomposable graphs as a function of the number of vertices.}
    \label{fig:asymtpotics}
\end{figure}

\begin{figure}[!ht]
    \centering
    \includegraphics[width=0.6\textwidth]{{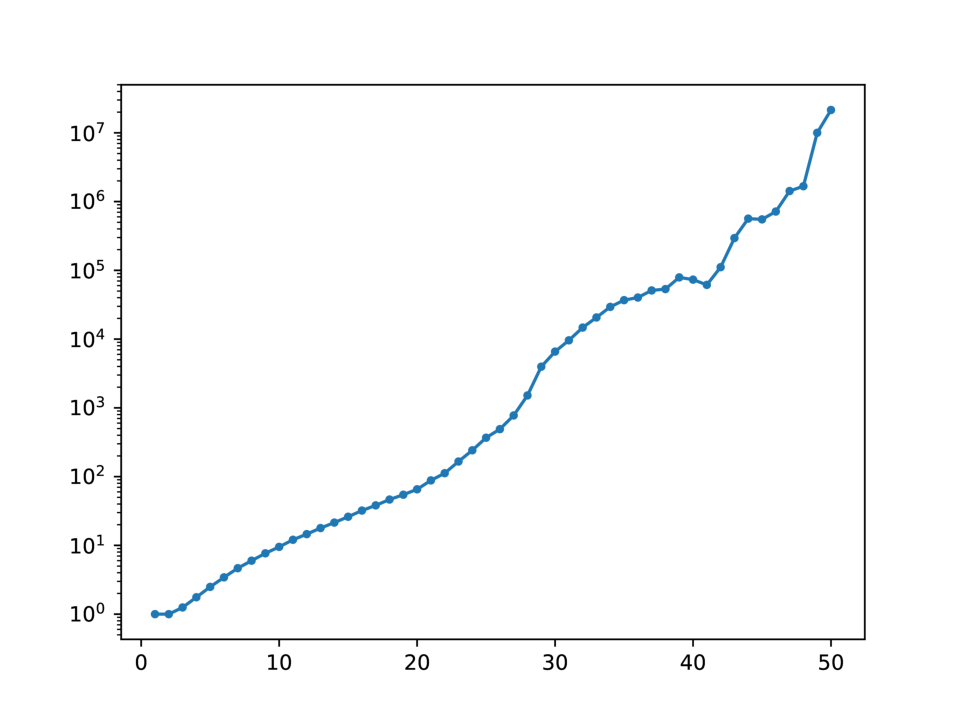}}
    \caption{Estimates of the expected number of junction trees per decomposable graph.}
    \label{fig:treespergraph}
\end{figure}




%% file: discussion.tex

Several \mcmcabbr\, methods for approximating distributions on the space of decomposable graphs have been proposed in the literature.
Most of them suffer from mixing problem inherited by the local moves that are made to the graphs (or junction trees) in the Markov chain.
In this paper we have presented the \jteabbr\, and the \jtcabbr\, for stochastically generating and collapsing junction trees for decomposable graphs in a vertex-by-vertex fashion.
The Markovian nature of both procedures together enable for essentially different sampling techniques, namely importance sampling, and particle Gibbs; see \citep{olsson2019}.
The main benefit of these sampling strategies is a substantial improvement of the mixing properties of the Markov chain; this improvement is possible due to (exploiting) global moves on the space of decomposable graphs.
The price paid for the junction tree representation when analysing the decomposable graph space is the imposed computational burden for calculating the number of junction trees for each of the sampled graphs.
This burden has been alleviated by the factorization property derived in Theorem \ref{thm:update_mu}, allowing for a faster dynamic updates.

Using importance sampling techniques, we were able to provide unbiased estimates of the number of decomposable graphs  and junction trees with given number of vertices \(\p\), as it appears as normalising constant in distributions specifically designed for this purpose.

As an alternative approach to the \jteabbr\, which incrementally construct a junction tree by adding one vertex at a time to the underlying graph, one may suggest a method which operates directly on the space of decomposable graphs.
The main difficulty arising when designing such a scheme is to express the transition probabilities in a tractable form while maintaining the ability to generate the whole space of decomposable graphs.

We expect that tailored data structures for the junction tree implementation which respect the sequential nature of the algorithms could greatly increase the speed.
For example, when propagating the particles in the  \smcabbr\, algorithm, the junction trees are not altered but rather copied and then expanded.
The reason for this is that several tree must be able to stem from the same ancestor.

%% file: acknowledgements.tex
The research of T. Pavlenko was partially supported by the Swedish Research Council,  Grant C0595201. 
J. Olsson gratefully acknowledges support by the Swedish Research Council, Grant 2018-05230.
We are also thankful to Jim Holmstr\"om for sharing his
Python knowledge with us.

%% file: appendix.tex
\section{}
\label{appendix:a}
\setcounter{equation}{0}

\subsection{Stochastic breath-first tree traversal}
Algorithm~\ref{alg:sbfs} provides the detailed steps in the stochastic \bfftabbr\, algorithm outlined in Section~\ref{sub:random_sampling_of_subtrees}.
\begin{algorithm}
\KwIn{A tree $\tree=(\graphnodeset,\graphedgeset)$, parameters $(\subtreecont, \subtreenonempty) \in (0,1)^2$}
\KwOut{$\tree' $ subtree of $\tree$}
    draw $u\sim \mathsf{Bernoulli}(\subtreenonempty)$\;
    \eIf{$u=0$}{
        \Return{$(\emptyset, \emptyset)$}
    }{
        draw $\gennodea \sim \unif[\graphnodeset]$\;
        $q\vcentcolon =[]$; \tcp{empty queue of vertices to visit}
        $q.enqueue(\gennodea)$\;
        $v\vcentcolon =[]$; \tcp{empty list of visited nodes}
        \While{$q\ne \emptyset$}{
            $\gennodea\vcentcolon =q.dequeue()$\;
            $v.append(\gennodea)$\;
            \For{$\gennodeb\in \neig{\gennodea} \setminus v$}{
                    draw $u\sim \mathsf{Bernoulli}(\alpha)$\; 
                    \If{$u=1$}{
                        $q.enqueue(\gennodeb)$\;
                    }
            }
        }
        let $\tree' = \tree \lbrack v \rbrack$\;
        \Return{$\tree'$}\;
    }
    \caption{}
\caption{Stochastic breadth-first tree traversal.}
\label{alg:sbfs}
\end{algorithm}

\subsection{Junction-tree expander (detailed steps)}
Below follows a more detailed description of the \jteabbr. 
The full algorithm is given in Algorithm~\ref{alg:cta}.
\begin{algorithm}
\caption{Junction-tree expander.}
\label{alg:cta}
\KwIn{$\tree \in \trsp[\jtind]$}
\KwOut{$\tree[+] \in \trsp[\jtind+1]$}
    \tcp{\(Step~1.\)}
    draw $\tree^\prime = (\graphnodeset^\prime, \graphedgeset^\prime) \sim \subtreeker{\subtreecont,\subtreenonempty}(\tree, \cdot)$\; \label{alg:cta:draw_subtree} 
    let $\tree[+]$ be a copy of $\tree$\; \label{alg:cta:copy_T}
    \eIf{$\tree' = (\emptyset, \emptyset)$}{
        \tcp{\(Step~2.\)}
        add a link from $\{\gennodea[\jtind+1]\}$ to one of the nodes in $\tree[+]$\;  \label{alg:cta:connect_isolated}
        \tcp{\(Step~3.\)}
        cut $\tree[+]$ at the separator $\emptyset$\;  
        \tcp{\(Step~4.\)}
        randomly connect \(\tree[+]\) into a forest\;
        \label{alg:cta:randomize_at_emptyset}
        \Return{$\tree[+]$}\;
    }
    {
    \tcp{\(Step~2^*.\)}
        enumerate the nodes in $\tree'$ by $\jtnode_1, \dots, \jtnode_\jtsubtreeorder$\; \label{alg:cta:enumerate_cliques}
        \For{$\jtnodeind=1 \to \jtsubtreeorder$ \label{alg:cta:create_nodes_loop}}{
            set $\jtnodeseps{\jtnodeind} \gets \bigcup\limits_{\jtnode\in \neig[\tree']{\jtnode_{\jtnodeind}}} \jtsep{\jtnode_\jtnodeind}{ \jtnode}$\; \label{alg:cta:label_sepset}
            \eIf{\mbox{there exists some}
                $\jtnode \in \neig[\tree']{\jtnode_{\jtnodeind}} \mbox{ such that } \jtnodeseps{\jtnodeind} = \jtsep{\jtnode_\jtnodeind}{\jtnode}$
            }{                
                draw $\jtnodearb{\jtnodeind} \sim \unif[{\potset(\jtnode_{\jtnodeind}\setminus \jtnodeseps{\jtnodeind}) \setminus \{\emptyset\}}]$\; \label{alg:cta:draw_arb}
            }{
                draw $\jtnodearb{\jtnodeind} \sim \unif[{\potset(\jtnode_{\jtnodeind}\setminus \jtnodeseps{\jtnodeind})}]$\; \label{alg:cta:draw_arb_nonempty}
            }
            set $ \jtnodearbrest{\jtnodeind} \gets \jtnodeseps{\jtnodeind} \cup \jtnodearb{\jtnodeind}$ and $\jtnewnode{\jtnodeind} \gets \jtnodearbrest{\jtnodeind} \cup \{\gennodea[\jtind+1]\}$\; \label{alg:cta:label_cliques}
            add $\jtnewnode{\jtnodeind}$ to $\tree[+]$\; \label{alg:cta:add_new_clique}
            \If {$\jtnodearbrest{\jtnodeind} = \jtnode_{\jtnodeind}$}{
                remove $\jtnode_{\jtnodeind}$ and its incident links from $\tree[+]$\; \label{alg:cta:swallowed_remove_neigs}
            }
        }
            \tcp{\(Step~3^*.\)}
        \ForEach{link $(\jtnode_{\jtnodeind}, \jtnode_{k})$ in $\tree'$ \label{alg:cta:isomorph_tree_loop}}{
            remove $(\jtnode_{\jtnodeind}, \jtnode_{k})$ from $\tree[+]$; \tcp{if not removed on Line \ref{alg:cta:swallowed_remove_neigs}} \label{alg:cta:remove_edge}
            add $(\jtnewnode{\jtnodeind}, \jtnewnode{k})$ to $\tree[+]$\; \label{alg:cta:add_edge}
        }

        \For{$\jtnodeind = 1 \to \jtsubtreeorder$}{
            \If {$\jtnodearbrest{\jtnodeind} = \jtnode_{\jtnodeind}$}{
                let $\neig[\tree]{\jtnode_{\jtnodeind}} \setminus \graphnodeset'$ be neighbors of  $\jtnewnode{\jtnodeind}$ in $\tree[+]$\;
                \label{alg:cta:transfer_edges_if_swallowed}
            }
        }
        \For{$\jtnodeind = 1 \to \jtsubtreeorder$ \label{alg:cta:connect_subtree_loop}}{
            \If {$\jtnewnode{\jtnodeind}$ \text{and} $\jtnode_{\jtnodeind}$ \text{are nodes in} $\tree[+]$}{
                add the link $(\jtnode_{\jtnodeind},\jtnewnode{\jtnodeind})$ to $\tree[+]$\; \label{alg:cta:connect_subtrees}
            }
        }
            \tcp{\(Step~4^*.\)}
        \For{$\jtnodeind=1 \to \jtsubtreeorder$ \label{alg:cta:for_random_neigs}}{
            \If{$\jtnodearbrest{\jtnodeind} \ne \jtnode_{\jtnodeind}$}{
                draw $\jtnodestomove{\jtnodeind} \sim \unif[{\potset(\{\jtnode \in \neig[{\tree[+]}]{\jtnode_{\jtnodeind}} : \sep_{ \jtnode, \jtnode_{\jtnodeind}} \subseteq \jtnodearbrest{\jtnodeind} \})}]$\; \label{alg:cta:draw_random_neighbors}
                move the neighbors $\jtnodestomove{\jtnodeind}$ of $\jtnode[\jtnodeind]$ to be neighbors of $\jtnewnode{\jtnodeind}$ instead\; \label{alg:cta:move_random_neighbors}
            }
        }
        \Return{$\tree[+]$}
    }
\end{algorithm}
\subsubsection*{Step~1: Subtree simulation}

In this step, a random subtree $\tree' = (V', E')$ of $\tree$ is sampled from $\subtreeker{\subtreecont,\subtreenonempty}(\tree, \cdot)$ (Line~\ref{alg:cta:draw_subtree}). After this, a new tree $\tree[+]$ is initiated as a copy of $\tree$ (Line~\ref{alg:cta:copy_T}), and all the manipulations described below refers to $\tree[+]$. Depending on whether $\tree'$ is empty or not, the algorithm proceeds in two substantially different ways.

\subsubsection*{Step~2: Node creation}

If $\tree'$ is empty, the new vertex $\gennodea[\jtind+1]$ is added as a node $\{\gennodea[\jtind+1]\}$ in its own and connected to one arbitrary existing node.
\subsubsection*{Step~3 and 4: Randomising the tree}

The tree is then cut at each link associated with the empty separator and reconstructed, a process we call \emph{randomisation at the separator} $\emptyset$ (Lines~\ref{alg:cta:connect_isolated}--\ref{alg:cta:randomize_at_emptyset}); see \autoref{appendix:b} or \citep{doi:10.1198/jcgs.2009.07129} for details. 
The randomisation step might seem superfluous at a first glance; however, it turns out to be needed in order to ensure that every junction tree has, as stated in Theorem \ref{thm:characterization}, a positive probability of being produced by iterative application of the algorithm.

\subsubsection*{Step~2\(^\ast\): Node creation}
If $\tree'$ is nonempty, the idea is to replicate its structure so that at the end of the algorithm, a subtree $\tree[+]'$ of $\tree[+]$ has been created where every node contains $\gennodea[\jtind+1]$. More specifically, for each node $\jtnode_\jtnodeind$, $\jtnodeind \in \{1,\dots,\jtsubtreeorder\}$, in $\tree'$, a new node $\jtnewnode{\jtnodeind}$ is created by connecting $\gennodea[\jtind+1]$ to a subset of $\jtnode_\jtnodeind$ while ensuring that the decomposability of $\trgr(\tree[+])$ is still maintained. If $\tree'$ has more than one node, it is, for each $j$, in order to avoid that a 4-cycle is formed in $\trgr(\tree[+])$, necessary to connect $\gennodea[\jtind+1]$ to all vertices in $\jtnodeseps{\jtnodeind} \eqdef \bigcup_{\jtnode\in \neig[\tree']{\jtnode_{\jtnodeind}} } 
\jtsep{\jtnode_\jtnodeind}{ \jtnode}$. For the rest of the vertices in $\jtnode_{\jtnodeind}$, a subset $\jtnodearb{\jtnodeind}$ is sampled uniformly at random, and $\jtnewnode{\jtnodeind}$ is formed as the union of $\jtnodearb{\jtnodeind}$, $\jtnodeseps{\jtnodeind}$, and $\{\gennodea[\jtind+1]\}$ (Lines~\ref{alg:cta:label_sepset}--\ref{alg:cta:label_cliques}). In the case where $\jtnodeseps{\jtnodeind}$ is identical to one of the separators $\jtsep{\jtnode_\jtnodeind}{ \jtnode}$, $\jtnode\in \neig[\tree']{\jtnode_{\jtnodeind}}$, it is necessary that $\jtnodearb{\jtnodeind}$ is nonempty in order to prevent the new node from being engulfed by some of its neighbors in $\tree[+]'$ (Line~\ref{alg:cta:draw_arb_nonempty}).
In the case where $\gennodea[\jtind+1]$ is connected to every vertex in $\jtnode_{\jtnodeind}$, $\jtnode_{\jtnodeind}$ is replaced by $\jtnewnode{\jtnodeind}$ (Lines~\ref{alg:cta:add_new_clique}--\ref{alg:cta:swallowed_remove_neigs}). 

\subsubsection*{Step~3\(^\ast\): Structure replication}

Having created the new nodes $\left( \jtnewnode{\jtnodeind} \right)_{j = 1}^\jtsubtreeorder$, links will be added between $\jtnewnode{\jtnodeind}$ and $\jtnewnode{k}$ whenever there is a link between $\jtnode_{\jtnodeind}$ and $\jtnode_{k}$ in $\tree'$ (Line~\ref{alg:cta:add_edge}). In this case, the link between $\jtnode_{\jtnodeind}$ and $\jtnode_{k}$ is removed (Line~\ref{alg:cta:remove_edge}) in order to avoid a 4-cycle to be formed on Line~\ref{alg:cta:connect_subtrees}. 
By this measure, $\tree[+]'$ replicates the structure of $\tree'$.
In order to connect $\tree[+]$ into a tree, links are added between each pair of nodes $\jtnewnode{\jtnodeind}$ and $\jtnode_{\jtnodeind}$ (Line~\ref{alg:cta:connect_subtrees}). 

\subsubsection*{Step~4\(^\ast\): Neighbor relocation}

Finally, we observe that for all $\jtnodeind \in \{1,\dots,\jtsubtreeorder\}$, any potential neighbor $\jtnode \in \neig[{\tree[+]}]{\jtnode_{\jtnodeind}}$ such that $\sep_{\jtnode, \jtnode_{\jtnodeind}} \subseteq \jtnodearbrest{\jtnodeind}$ 
can be moved to be a neighbor of $\jtnewnode{\jtnodeind}$ instead while maintaining the junction tree property (Lines~\ref{alg:cta:draw_random_neighbors}--\ref{alg:cta:move_random_neighbors}). 
In the special case where the node $\jtnode_{\jtnodeind}$ is substituted by $\jtnewnode{\jtnodeind}$, all the neighbors of $\jtnode_{\jtnodeind}$ will simply be neighbors of $\jtnewnode{\jtnodeind}$ instead (Line~\ref{alg:cta:transfer_edges_if_swallowed}).

\subsection{Junction-tree collapser (detailed steps)}
Below follows some more detailed description of the \jtcabbr. 
The full algorithm is given in Algorithm~\ref{alg:rcta}.
\begin{algorithm}[!ht]
     \KwIn{$\tree[+]=(\graphnodeset[+],\graphedgeset[+])\in \trsp[\jtind+1]$}
     \KwOut{$\tree \sim \bk{\jtind}(\tree[+], \cdot)$}
     let $\tree$ be a copy of $\tree[+]$\;
     \eIf{$\{\gennodea[\jtind+1]\} \in \graphnodeset[+]$ \label{alg:rcta:if_isloated}}{
        remove $\{\gennodea[\jtind+1]\}$ and its incident edges from $\tree$\;   \label{alg:rcta:remove_node}
        connect $\tree$ into a tree using Algorithm \ref{alg:jt_shuffling} (Appendix \ref{appendix:b})\; \label{alg:rcta:connect_forest}
     }{
     let $\tree[+]'$ be the subtree of nodes in $\tree[+]$ containing $\gennodea[\jtind+1]$\; 
     enumerate the nodes in $\tree[+]'$ by $\jtnewnode{1},\dots, \jtnewnode{|\graphnodeset[+]'|}$\;
        \For{$\jtnodeind \rightarrow 1,\dots,|\graphnodeset[+]'|$}{
         let $\jtpotorigins{\jtnodeind} \gets \{\jtnode \in \neig[{\tree[+]}]{\jtnewnode{\jtnodeind}}  : \sep_{\jtnode,\jtnewnode{\jtnodeind}} = \jtnewnode{\jtnodeind} \setminus \{\gennodea[\jtind+1]\} \}$\; \label{alg:rcta:set_potorigins}
            \eIf{$\jtpotorigins{\jtnodeind}=\emptyset$}{
                $\jtnode[\jtnodeind] \gets \jtnewnode{\jtnodeind}\setminus\{\gennodea[\jtind+1]\}$\; \label{alg:rcta:swallowed_origin}
            }{
                draw $\jtnode[\jtnodeind] \sim \unif[\jtpotorigins{\jtnodeind}]$\; \label{alg:rcta:random_origin}
            }
            let $\neig[{\tree[+]}]{\jtnewnode{\jtnodeind}}\setminus \jtnode[\jtnodeind]$ be neighbors of $\jtnode[\jtnodeind]$ in $\tree$\; \label{alg:rcta:set_neighbors}
            remove $\jtnewnode{\jtnodeind}$ and its incident links from $\tree$\; \label{alg:rcta:remove_old_node}
        }
     }
     \Return{T}\;
     \caption{Junction-tree collapser.}
     \label{alg:rcta}
\end{algorithm}

Similarly to the \jteabbr, the \jtcabbr\, takes two different forms depending on whether ${\{\gennodea[\jtind+1]\}}$ is present as a node in $\tree[+]$ or not.
Specifically, if $\{\gennodea[\jtind+1]\}\in\graphnodeset[+]$, then $\{\gennodea[\jtind+1]\}$ is removed from $\tree[+]$ and the resulting forest is reconnected uniformly at random (Lines~\ref{alg:rcta:if_isloated}--\ref{alg:rcta:connect_forest} in Algorithm~\ref{alg:rcta}).

Otherwise, if $\{\gennodea[\jtind+1]\}\notin\graphnodeset[+]$ we denote by  $\{\jtnewnode{\jtnodeind}\}_{\jtnodeind=1}^{|\graphnodeset[+]'|} $ the nodes in the subtree $\tree[+]'$ induced by the nodes containing the vertex $\gennodea[\jtind+1]$. The aim is now to identify the nodes that can serve as a subtree in Algorithm~\ref{alg:cta} to produce $\tree[+]$. 
Since each node in the subtree sampled initially in Algorithm $\ref{alg:cta}$ will give rise to a new node, it is enough to determine, for each $j \in \{1,\dots,\jtsubtreeorder\}$, the node $\jtnode[\jtnodeind]$ that can be used for producing $\jtnewnode{\jtnodeind}$ (reversing Lines~\ref{alg:cta:create_nodes_loop}--\ref{alg:cta:swallowed_remove_neigs} in Algorithm~\ref{alg:cta}).
For each $\jtnodeind$, we define a set of candidate nodes $\jtpotorigins{\jtnodeind}=\{\jtnode \in \graphnodeset[+]: \jtnewnode{\jtnodeind}\cap\jtnode=\jtnodearbrest{\jtnodeind}\}$.
If $\jtpotorigins{\jtnodeind}=\emptyset$,
we let $\jtnode[\jtnodeind]=\jtnewnode{\jtnodeind}\setminus \{\gennodea[\jtind+1]\}$ (Line~\ref{alg:rcta:swallowed_origin} in Algorithm~\ref{alg:rcta}).
Otherwise, $\jtnode[\jtnodeind]$ is drawn at random from the uniform distribution over $\jtpotorigins{\jtnodeind}$ (Line \ref{alg:rcta:random_origin}).
In either case, the edges incident to $\jtnewnode{\jtnodeind}$ are moved to $\jtnode[\jtnodeind]$ (Line~\ref{alg:rcta:set_neighbors}).
\subsection{Proofs and lemmas}
\begin{mylem}

    \label{lem:sep_cond}
    Let $\tree$ be a tree where each node is a subset of some finite set.
    Then $\tree$ satisfies the junction tree property if and only if for any path $\jtnode_{1} \sim  \jtnode_{\pathorder} =( \jtnode_{1},\dots,  \jtnode_{\pathorder})$ in $\tree$ it holds that
    \begin{align*}
        \jtnode_{1}\cap \jtnode_{\pathorder} \subset \bigcap\limits_{\jtnodeind=2}^\pathorder \jtsep{\jtnode_{{\jtnodeind-1}}}{\jtnode_{{\jtnodeind}}}.
    \end{align*}

\end{mylem}

\begin{proof}

    The statement of the lemma follows by noting that
    \begin{align*}
        \bigcap\limits_{\jtnodeind=1}^\pathorder \jtnode_{\jtnodeind} =  \bigcap\limits_{\jtnodeind=2}^\pathorder (\jtnode_{{\jtnodeind-1}}\cap \jtnode_{\jtnodeind}) = \bigcap\limits_{\jtnodeind=2}^\pathorder \jtsep{\jtnode_{{\jtnodeind-1}} }{\jtnode_{{\jtnodeind}}},
    \end{align*}
    which implies that
   \begin{align*}
        \jtnode_{1}\cap \jtnode_{\pathorder} \subset \bigcap\limits_{\jtnodeind=1}^\pathorder \jtnode_{\jtnodeind} \iff  \jtnode_{1}\cap \jtnode_{\pathorder} \subset \bigcap\limits_{\jtnodeind=2}^\pathorder \jtsep{{\jtnode_{{\jtnodeind-1}}}}{\jtnode_{{\jtnodeind}}}.
    \end{align*}

\end{proof}

\begin{proof}[Proof of Theorem \ref{thm:jt_preservation}]

    We prove this theorem by taking a generative perspective in the sense that we rely on the sampling procedure of $\prop{\jtind}(\tree,\cdot)$ given by Algorithm \ref{alg:cta}.
    We also adopt the same notation as in Algorithm \ref{alg:cta}.

    In order to prove \ref{itm:cta:jt_preservation_prop} we assume that $\tree[+]$ is generated by Algorithm \ref{alg:cta} with input $\tree$ and show that $\tree[+]\in \trsp[\jtind+1]$ by going through the algorithm in a step-by-step fashion.
    At Line~\ref{alg:cta:draw_subtree} a subtree $\tree'$ is drawn.
    We treat the cases $\tree'=(\emptyset, \emptyset)$ and $\tree'\ne(\emptyset, \emptyset)$ separately.

    First, assume that $\tree'=(\emptyset,\emptyset)$. 
    Since the node $\{\gennodea[\jtind+1]\}$ does not intersect any other node in $\tree$, it can be connected to an arbitrary node  with separator $\emptyset$ without violating the junction tree property (Line~\ref{alg:cta:connect_isolated}). 
    In addition, \citet{Thomas:2009aa} show that randomising a tree at a given separator preserves the junction tree property (Line~\ref{alg:cta:randomize_at_emptyset}).

    For $\tree'\ne(\emptyset,\emptyset)$, we first show that $\tree[+]$ produced on Lines~\ref{alg:cta:enumerate_cliques}--\ref{alg:cta:connect_subtrees} is a tree that satisfies the junction tree property.
    Indeed, $\tree[+]$ is a tree since the subtrees produced up to Line~\ref{alg:cta:transfer_edges_if_swallowed} are all reconnected through the same tree $\tree[+][\{\jtnewnode{\jtnodeind}\}_{\jtnodeind=1}^\jtsubtreeorder]$ by the operations on Lines~\ref{alg:cta:connect_subtree_loop}--\ref{alg:cta:connect_subtrees}.
    To ensure the junction tree property of $\tree[+]$, consider a general path 
    \begin{equation*}
        \label{eq:path_in_tpp}
        (\genseta_{1},\dots,\genseta_{\pathordera}, \jtnode_{\subseqind_1} ,\jtnewnode{\subseqind_1},\dots ,\jtnewnode{\subseqind_\pathorderb} , \jtnode_{\subseqind_\pathorderb}, \gensetb_1,\dots, \gensetb_{\pathorderc})
    \end{equation*}
    passing through $\tree[+]'$ in $\tree[+]$, where
     $(\genseta_\jtnodeind)_{\jtnodeind=1}^{\pathordera}$ and $
     (\gensetb_\jtnodeind)_{\jtnodeind=1}^{\pathorderb}$
    are nonempty sequences of nodes which also belong to $\tree$. 

    The fact that 
    $(\jtnewnode{\subseqind_\jtnodeind})_{\jtnodeind=1}^{\pathorderb}$
    is the $\jtnewnode{\subseqind_1}$-$\jtnewnode{\subseqind_\pathorderb}$ path in $\tree[+]$ implies that 
     $(\jtnode_{\subseqind_\jtnodeind})_{\jtnodeind=1}^{\pathorderb}$
    is the  $\jtnode_{\subseqind_1}$-$\jtnode_{\subseqind_\pathorderb}$ path in $\tree$, since, by construction (Lines~\ref{alg:cta:isomorph_tree_loop}--\ref{alg:cta:add_edge}),
    $(\jtnode_\jtnodeind,\jtnode_k) \in \graphedgeset^\prime$ if and only if $(\jtnewnode{\jtnodeind},\jtnewnode{k}) \in \graphedgeset[+]$.
    Thus,
    \begin{align*}
        (\genseta_{1},\dots,\genseta_{\pathordera},\jtnode_{\subseqind_1},\dots ,\jtnode_{\subseqind_\pathorderb},\gensetb_1,\dots, \gensetb_{\pathorderc})
    \end{align*}
    is the $\genseta$-$\gensetb$ path in $\tree$.
    The junction tree property of $\tree$ ensures that $\genseta \cap \gensetb \subset I_{\genseta\sim \gensetb}$, where
    \begin{equation} \label{eq:abintersection}
        I_{\genseta\sim \gensetb} \eqdef \left( \bigcap_{\jtnodeind=1}^{\pathordera}\genseta_{\jtnodeind}\right) \cap \left( \bigcap_{\jtnodeind=1}^\pathorderb \jtnode_{\subseqind_\jtnodeind} \right) \cap \left( \bigcap_{\jtnodeind=1}^{\pathorderc}\gensetb_\jtnodeind \right).
    \end{equation}
    Now, consider the intersection 
    \begin{align*}
        I^+_{\genseta\sim \gensetb} \eqdef \left( \bigcap_{\jtnodeind=1}^{\pathordera}\genseta_{\jtnodeind} \right) \cap \jtnode_{\subseqind_1} \cap \left( \bigcap_{\jtnodeind=1}^\pathorderb \jtnewnode{\subseqind_\jtnodeind} \right) \cap \jtnode_{\subseqind_\pathorderb} \cap \left( \bigcap_{\jtnodeind=1}^{\pathorderc}\gensetb_\jtnodeind \right)
    \end{align*}
    of the nodes in $\genseta\sim \gensetb$ in $\tree[+]$.
    For $\pathorderb=1$, it holds that $\jtnewnode{\subseqind_1} = \jtnode_{\subseqind_1} \cup \{\gennodea[\jtind+1]\}$, corresponding to the case where $\jtnode_{\subseqind_1}$ is engulfed into $\jtnewnode{\subseqind_1}$. 
    For $\pathorderb \ge 2$,
         the junction tree property in $\tree$ ensures via Lemma~\ref{lem:sep_cond} that 
    \begin{align*}
        \bigcap_{\jtnodeind=1}^\pathorderb \jtnewnode{\subseqind_\jtnodeind} = & \bigcap_{\jtnodeind=2}^\pathorderb \jtsep{ \jtnewnode{\subseqind_{\jtnodeind-1} }}{\jtnewnode{\subseqind_\jtnodeind} }
        =  \left (\bigcap_{\jtnodeind=2}^\pathorderb \jtsep{\jtnode_{\subseqind_{\jtnodeind-1}}}{\jtnode_{\subseqind_\jtnodeind}}\right ) \cup \{\gennodea[\jtind+1]\}
        =  \left (\bigcap_{\jtnodeind=1}^\pathorderb \jtnode_{\subseqind_\jtnodeind}\right) \cup \{\gennodea[\jtind+1]\}.
    \end{align*}
    It hence holds that $\genseta\cap \gensetb \subset I_{\genseta\sim \gensetb} \subset I^+_{\genseta\sim \gensetb}$.

    Now, consider the final version of $\tree[+]$ obtained after the relocation step on Lines~\ref{alg:cta:for_random_neigs}--\ref{alg:cta:move_random_neighbors}. 
    Let $\jtnode\in \jtnodestomove{\jtnodeind}$ and let $\tree_{\jtsep{\jtnode_{\jtnodeind}}{\jtnode}}$ be the subtree of $\tree[+]$ induced by the nodes containing the separator $\jtsep{\jtnode_{\jtnodeind}}{\jtnode}$.
    In addition to $\jtnode$ and $\jtnode_{\jtnodeind}$, it is clear that $\jtnewnode{\jtnodeind}$ is also a node in $\tree_{\jtsep{\jtnode_{\jtnodeind}}{\jtnode}}$ since $\jtsep{\jtnode_{\jtnodeind}}{\jtnode} \subseteq \jtsep{\jtnewnode{\jtnodeind}}{\jtnode}$.
    Now the fact that $\jtsep{\jtnode_{\jtnodeind}}{\jtnode} \subseteq \jtsep{\jtnewnode{\jtnodeind}}{\jtnode}$ also implies that the tree obtained by letting $\jtnode$ be a neighbor of $\jtnewnode{\jtnodeind}$ instead of $\jtnode_{\jtnodeind}$ also satisfies the junction tree property by \citet{doi:10.1198/jcgs.2009.07129}.

    Finally, \ref{itm:cta:jt_ind_subgraph_prop} follows directly since the only new vertex added to $\trgr(\tree)$ in order to get $\trgr(\tree[+])$ is $\gennodea[\jtind+1]$ and no edges have been removed between the vertices
    \(\{\gennodea[\jtnodeind]\}_{\jtnodeind=1}^\jtind\).

\end{proof}

\begin{proof}[Proof of Theorem~\ref{thm:characterization}]
 
    In this proof we use the property \ref{itm:rcta:valid_cta_expansion} of $\bk{\jtind}$ provided by Theorem~\ref{thm:rcta} and proved independently below.
 
    The space containing the trivial junction tree is $\trsp[1]=\{(\{\gennodea[1]\},\emptyset)\}$.
    We proceed by induction over the number of vertices.
    For the base case $\jtind = 2$, $\trsp[2]=\{\tree_1, \tree_2\}$, where $\tree_1=(\{\{\gennodea[1],\gennodea[2]\}\},\emptyset)$ is the unique tree constructed from $(\{\{\gennodea[1]\}\}, \emptyset)$ via the subtree $(\{\{\gennodea[1]\}\}, \emptyset)$ and $\tree_2=(\{\{\gennodea[1]\},\{\gennodea[2]\}\}, \{(\{\gennodea[1]\},\{\gennodea[2]\})\} )$ is the unique tree constructed from $(\{\{\gennodea[1]\}\}, \emptyset)$ via the subtree $(\emptyset,\emptyset)$.

    For $m \geq 3$, assume inductively that $\supp(\propmarg{\jtind-1}(\cdot)) = \trsp[\jtind-1]$ and let $\tree \in \trsp[\jtind]$ be an arbitrary junction tree. 
    It suffice to show that there exists a junction tree $\tree[-]\in \trsp[\jtind-1]$ such that $\prop{\jtind-1}(\tree[-], \tree)>0$ since then $\propmarg{\jtind}(\tree) \ge \propmarg{\jtind-1}(\tree[-])\prop{\jtind-1}(\tree[-], \tree) >0$. 
    But this follows directly by drawing any $\tree[-]\sim \bk{\jtind-1}(\tree,\cdot)$ since $\supp(\bk{m-1}(\tree, \cdot)) \subseteq \supp(\prop{m-1}(\cdot,\tree))$ (by \ref{itm:rcta:valid_cta_expansion} in Theorem~\ref{thm:rcta}), meaning that $\tree\in \supp(\prop{m-1}(\tree[-],\cdot))$.
    Thus, every junction tree in $\trsp[m]$ can be constructed, and we conclude the proof by induction. 

\end{proof}

\begin{proof}[Proof of Proposition~\ref{thm:uniqueness}]

    In this proof, we take a generative perspective in the sense that we rely on the sampling procedure given by Algorithm \ref{alg:cta} and regard $\tree[+]$ as an expansion of $\tree$.
    We further adopt the same notation as in Algorithm \ref{alg:cta} when possible.

    Let 
    $\graphnodeset[+]'=\{\jtnewnode{\jtnodeind}\}_{\jtnodeind=1}^{|\graphnodeset[+]^\prime|}$ 
    be the set of nodes in $\tree[+]$ containing the vertex $\gennodea[\jtind+1]$. 
    The induced subgraph $\tree[+]\lbrack\graphnodeset[+]'\rbrack$ will necessarily be a subtree of $\tree[+]$ (see e.g. \citet{ChordalGraphs}), which we denote by $\tree[+]^\prime=(\graphnodeset[+]^\prime,\graphedgeset[+]^\prime)$.
    For each $\jtnodeind\in \{1,\dots,|\graphedgeset[+]^\prime|\}$, we define a set $\jtpotorigins{\jtnodeind}=\{\jtnode \in \graphnodeset[+]: \jtnewnode{\jtnodeind}\cap\jtnode=\jtnodearbrest{\jtnodeind}\}$, which we interpret as the candidate nodes in $\graphnodeset$ from which each $\jtnewnode{\jtnodeind}=\jtnodearbrest{\jtnodeind} \cup \{\gennodea[\jtind+1]\} \in \graphnodeset[+]$ could potentially have emerged.
    We distinguish between two main situations for $\mathsf{\tree}(\tree, \tree[+])$ depending on $|\graphnodeset[+]^\prime|$.

    For $|\graphnodeset[+]^\prime|=1$,
    \begin{itemize}
        \setlength\itemsep{0.5em}
        \item 
            If $|\jtpotorigins{1}| = 0$, then $\jtnode_{1}=\jtnewnode{1}\setminus \{\gennodea[\jtind+1]\}$ so that $\mathsf{\tree}(\tree, \tree[+])= \{ (\{\jtnode_{1}\}, \emptyset)\}$ due to Lines~\ref{alg:cta:add_new_clique}--\ref{alg:cta:swallowed_remove_neigs}. 
        \item 
            If $|\jtpotorigins{1}| = 1$, then clearly $\mathsf{\tree}(\tree, \tree[+])=\{(\{\jtnode_1\},\emptyset)\}$, where $\{\jtnode_1\} = \jtpotorigins{1}$
            due to Lines~\ref{alg:cta:connect_subtree_loop}--\ref{alg:cta:connect_subtrees}. 
        \item 
            If $|\jtpotorigins{1}| \ge 2$, consider the enumeration \(\jtpotorigins{1}=\{\tilde\jtnode_{\jtnodeind}\}_{\jtnodeind=1}^{|\jtpotorigins{1}|}\). 
            The set $\mathsf{\tree}(\tree, \tree[+])$ is clearly non-empty, thus we can assume that $(\{\tilde\jtnode_{1}\},\emptyset)\in\mathsf{\tree}(\tree, \tree[+])$.
            Note that, since $\jtpotorigins{1}$ consists of more than one element, all nodes in
            \(\{\tilde\jtnode_{\jtnodeind}\}_{\jtnodeind=2}^{|\jtpotorigins{1}|}\)
            are former neighbors of $\tilde\jtnode_1$ by Lines~\ref{alg:cta:for_random_neigs}--\ref{alg:cta:move_random_neighbors}.
            Thus, every node in
             \(\{\tilde\jtnode_{\jtnodeind}\}_{\jtnodeind=2}^{|\jtpotorigins{1}|}\)
            are neighbors of $\tilde\jtnode_1$ in $\tree$.
            This implies that for $|\jtpotorigins{1}|=2$, the subtree could also be  $(\{\tilde\jtnode_2\},\emptyset)$ since the $\tilde\jtnode_1$ could be moved at Lines~\ref{alg:cta:for_random_neigs}--\ref{alg:cta:move_random_neighbors}. Thus $\mathsf{\tree}(\tree, \tree[+]) = \{(\{\tilde\jtnode_1\}, \emptyset), (\{\tilde\jtnode_2\}, \emptyset)\}$.
            For $|\jtpotorigins{1}| \ge 3$, $(\{\tilde\jtnode_1\}, \emptyset)$ is necessarily the unique subtree in $\mathsf{\tree}(\tree, \tree[+])$ since if there would exist another subtree $(\{\tilde\jtnode_2\}, \emptyset)$, both $\tilde\jtnode_1$ and $\tilde\jtnode_2$ would have $\tilde\jtnode_3$ as neighbor in $\tree$, which would form a cycle.
    \end{itemize}

    For $|\graphnodeset[+]^\prime|\ge 2$, by construction, for each link  $(\jtnewnode{\jtnodeind},\jtnewnode{\jtnodeindb})\in \graphedgeset[+]'$, where $\jtnewnode{\jtnodeind} = \jtnodearbrest{\jtnodeind} \cup \{\gennodea[\jtind+1]\}$ and $\jtnewnode{\jtnodeindb} = \jtnodearbrest{\jtnodeindb} \cup \{\gennodea[\jtind+1]\}$ we can associate a link $(\jtnode_{\jtnodeind}, \jtnode_{\jtnodeindb})\in \graphedgeset$, where $\jtnode_{\jtnodeind}=\jtnodearbrest{\jtnodeind} \cup \jtnoderest{\jtnodeind}$ and $\jtnode_{\jtnodeindb}=\jtnodearbrest{\jtnodeindb} \cup \jtnoderest{\jtnodeindb}$ are emerging nodes in $\tree$ and $ \jtnoderest{\jtnodeind}$ and $ \jtnoderest{\jtnodeindb}$ may be empty sets.
    Thus we can form the subtree $\tree'=(\graphnodeset', \graphedgeset')$ which we regard as the subtree in Algorithm \ref{alg:cta} (Line~\ref{alg:cta:draw_subtree}), where $\graphnodeset'=\{\jtnode_{\jtnodeind}\in \graphnodeset: \jtnewnode{\jtnodeind} \in \graphnodeset[+]'\}$ and $\graphedgeset' = \{(\jtnode_\jtnodeind, \jtnode_\jtnodeindb)\in \graphedgeset:(\jtnewnode{\jtnodeind}, \jtnewnode{\jtnodeindb}) \in \graphedgeset[+]'\}$.
    Now, suppose that there exists another subtree $\tree''=(\graphnodeset'', \graphedgeset'')$, isomorphic to $\tree'$, where $\graphnodeset''=\{\jtnode_{\jtnodeind}'=\jtnodearbrest{\jtnodeind} \cup \jtnoderest{\jtnodeind}': \jtnewnode{\jtnodeind} \in \graphnodeset[+]''\}$, $\jtnode_{\jtnodeind}'\neq \jtnode_{\jtnodeind}$ for some $\jtnodeind$ and $\graphedgeset'' = \{(\jtnode_\jtnodeind', \jtnode_\jtnodeindb')\in \graphedgeset:(\jtnewnode{\jtnodeind}, \jtnewnode{\jtnodeindb}) \in \graphedgeset[+]'\}$.
    Enumerate the nodes such that $\jtnode_{1}'\ne\jtnode_{1}$ and let for simplicity $\jtnode_{2}'=\jtnode_{2}\in \neig{\jtnode_{1}'}$.
    Then, since the neighbors of $\jtnewnode{1}$ except for $\jtnode_{1}$ are neighbors of $\jtnode_{1}$ in $\tree$, the link $(\jtnode_{1},\jtnode_{1}')$ would be present in $\tree$. 
    Also the link $(\jtnode_{1},\jtnode_{2})$ is present in $\tree$ since $\tree'$ is a subtree of $\tree$. 
    Similarly, since $\tree''$ is a subtree of $\tree$, the link $(\jtnode_{1}',\jtnode_{2})$ would also be present in $\tree$. 
    Thus we would have a 3-cycle in $\tree$ which contradicts the assumption of $\tree$ being a tree. 
    Thus $\mathsf{\tree}(\tree, \tree[+])=\{\tree'\}$.

\end{proof}

\begin{proof}[Proof of Theorem \ref{thm:rcta}]
        We prove this theorem by taking a generative perspective in the sense that we rely on the sampling procedures of $\prop{\jtind}$ and $\bk{\jtind+1}$ given by Algorithm \ref{alg:cta} and \ref{alg:rcta} respectively.
    We also adopt the same notation as in these algorithms.
        To show \ref{itm:rcta:jt_preservation_prop} and \ref{itm:rcta:valid_cta_expansion} we distinguish between the cases $\{\gennodea[\jtind+1]\}\in \graphnodeset[+]$ and $\{\gennodea[\jtind+1]\}\notin \graphnodeset[+]$.
        For both cases, we let $\tree\in \supp(\bk{\jtind+1}(\tree[+],\cdot))$.
        We prove \ref{itm:rcta:jt_preservation_prop} by following the steps in Algorithm \ref{alg:rcta} with input $\tree[+]$. 
        For \ref{itm:rcta:valid_cta_expansion}, we show that $\tree[+]$ could be obtained by Algorithm \ref{alg:cta} with input $\tree$.

    If $\{\gennodea[\jtind+1]\}\in \graphnodeset[+]$, then no other node in $\tree[+]$ will contain the vertex $\gennodea[\jtind+1]$ which in turn implies that each neighbor in $\neig[{\tree[+]}]{\{\gennodea[\jtind+1]\}}$ will have $\emptyset$ as associated separator.
    Removing one node from a tree will always result in a forest possibly containing only one tree. 
    Thus the removal of $\{\gennodea[\jtind+1]\}$ from $\tree$ on Line~\ref{alg:rcta:remove_node} will result in a forest.
    Since $\{\gennodea[\jtind+1]\}$ is not contained in any of the trees in the forest, these will all trivially satisfy the junction tree property and the connection of $\tree$ into a tree by Line~\ref{alg:rcta:connect_forest} will give a random junction tree for $\trgr(\tree[+])[\{\gennodea[\ell]\}_{\ell=1}^\jtind]$, which proves \ref{itm:rcta:jt_preservation_prop} in this case. 
    For the \ref{itm:rcta:valid_cta_expansion} part, we simply observe that $\tree[+]$ can be constructed from $\tree$ by first drawing the empty subtree on Line~\ref{alg:cta:draw_subtree} and then obtaining $\tree[+]$ at the randomization on Line~\ref{alg:cta:randomize_at_emptyset} in Algorithm \ref{alg:cta}.

    Now, assume that $\{\gennodea[\jtind+1]\} \notin \graphnodeset[+]$.
    We proceed by showing \ref{itm:rcta:jt_preservation_prop}, i.e. that $\tree\in \trsp[\jtind]$.
    We first show that $\tree$ is a tree.
    Since for every $\jtnodeind \in \{1, \ldots, |\graphnodeset[+]'|\}$, all elements in $\neig[{\tree[+]}]{\jtnewnode{\jtnodeind}} \setminus \jtnode[\jtnodeind]$ are set to be neighbors of $\jtnode[\jtnodeind]$ in $\tree$, 
        for all $\jtnodeind, \jtnodeindb \in \{1,\dots,|\graphnodeset[+]'|\}$ it follows that    
    \begin{align*}
        (\jtnewnode{\jtnodeind}, \jtnewnode{\jtnodeindb}) \in \graphedgeset[+] \iff (\jtnode_{\jtnodeind},\jtnode_{\jtnodeindb}) \in \graphedgeset. 
    \end{align*}
    Hence, since $\tree[+]'$ is a subtree of $\tree[+]$, $\tree'=(\graphnodeset', \graphedgeset')$ is a tree, where $\graphnodeset'=
    \{\jtnode_{\jtnodeind}\}_{\jtnodeind=1}^{|\graphnodeset[+]'|}$
    and $\graphedgeset'=\{(\jtnode_{\jtnodeind},\jtnode_{\jtnodeindb}) \in \graphedgeset : (\jtnewnode{\jtnodeind}, \jtnewnode{\jtnodeindb}) \in \graphedgeset[+] \}$.
    Further, since elements of $\neig[{\tree[+]}]{\jtnewnode{\jtnodeind}} \setminus \jtnode[\jtnodeind]$ are not mutual neighbors, and parts of distinct subtrees of $\tree$, $\tree'$ is a subtree of $\tree$.
    As a consequence, we may assume that an arbitrary path (of length at least 2) in $\tree$ is of form
    \begin{align*}
        (\genseta_1,\dots,\genseta_{\pathordera}, \jtnode_{\subseqind_1},\dots, \jtnode_{\subseqind_\pathorderb}, \gensetb_1,\dots,\gensetb_{\pathorderc}),
    \end{align*}
     where $\pathordera\ge 0$, $\pathorderb\ge 0$, $\pathorderc \ge 0$ and $\{\jtnode_{\subseqind_\jtnodeind}\}_{\jtnodeind=1}^{\pathorderb} \subseteq \graphnodeset'$. 
    Let $\genseta$ and $\gensetb$ be the first and last element in this path, respectively. 
    Let the intersection $I_{\genseta\sim \gensetb}$ be defined by  \eqref{eq:abintersection}. 
    We must prove that $\genseta\cap \gensetb \subset I_{\genseta\sim \gensetb}$.
    We know that in $\tree[+]$, the node $\genseta_{\pathordera}$ was connected to either $\jtnode_{\subseqind_1}$ (in which case $(\jtnode_{\subseqind_1},\jtnewnode{\subseqind_1})\in \graphedgeset[+]$) or to $\jtnewnode{\subseqind_1}$ and $\gensetb_1$ was connected to either $\jtnode_{\subseqind_\pathorderb}$ (in which case $(\jtnode_{\subseqind_\pathorderb},\jtnewnode{\subseqind_\pathorderb})\in \graphedgeset[+]$) or to $\jtnewnode{\pathorderb}$.
    First, assume that $\genseta_{\pathordera}$ was connected to $\jtnewnode{\subseqind_1}$ and $\gensetb_1$ was connected to $\jtnewnode{\pathorderb}$, then the $\genseta$-$\gensetb$ path in $\tree[+]$ is of form
    \begin{align*}
        (\genseta_1,\dots,\genseta_{\pathordera},\jtnewnode{\subseqind_1},\dots,\jtnewnode{\subseqind_\pathorderb}, \gensetb_1,\dots,\gensetb_{\pathorderc}).
    \end{align*}
    Let
    \begin{align*}
        I_{\genseta\sim \gensetb}^+ =
        \left(
        \bigcap \limits_{\jtnodeind=1}^{\pathordera} \genseta_\jtnodeind \right) \cap
        \left(\bigcap \limits_{\jtnodeind=1}^\pathorderb \jtnewnode{\subseqind_\jtnodeind} \right) \cap \left( \bigcap \limits_{\jtnodeind=1}^{\pathorderc} \gensetb_\jtnodeind. \right)
    \end{align*}
    We know that, since $\tree[+]$ is a junction tree, $\genseta\cap \gensetb \subset I_{\genseta\sim \gensetb}^+$.
    Moreover, by Lemma~\ref{lem:sep_cond} it holds that $\bigcap \limits_{\jtnodeind=1}^\pathorderb \jtnewnode{\subseqind_\jtnodeind} = \bigcap \limits_{\jtnodeind=1}^\pathorderb \jtnode_{\subseqind_\jtnodeind}\cup \{\gennodea[\jtind+1]\}$.
    But, $\gennodea[\jtind+1] \notin \genseta$ and $\gennodea[\jtind+1] \notin \gensetb$, thus $I_{\genseta\sim \gensetb}=I_{\genseta\sim \gensetb}^+$ so that $\genseta\cap \gensetb \subset I_{\genseta\sim \gensetb}$.
    Now, note that $I^+_{\genseta\sim \gensetb}\cap (\jtnode_{\subseqind_1} \cap \jtnode_{\subseqind_\pathorderb})=I^+_{\genseta\sim \gensetb}$, so that adding $\jtnode_{\subseqind_1}$ and $ \jtnode_{\subseqind_\pathorderb}$ to the path does not change anything, thus the junction tree property also holds in the case where $\genseta_{\pathordera}$ was connected to $\jtnode_{\subseqind_1}$ or $\gensetb_1$ was connected to $\jtnode_{\subseqind_\pathorderb}$.

    To show \ref{itm:rcta:valid_cta_expansion} in this case, observe that $\tree$ can be expanded to $\tree[+]$ by first drawing the subtree $\tree'$ on Line~\ref{alg:cta:draw_subtree}.
    Then, by identifying $\jtnodearbrest{\jtnodeind}=\jtnewnode{\jtnodeind}\cap \jtnode_{\jtnodeind}$ and $\jtnodeseps{\jtnodeind}=\bigcup\limits_{\jtnode \in \neig[{\tree[+]'}]{\jtnewnode{\jtnodeind}}}\sep_{\jtnewnode{\jtnodeind},\jtnode}$, there is a positive probability for obtaining $\jtnodearb{\jtnodeind}=\jtnodearbrest{\jtnodeind} \setminus \jtnodeseps{\jtnodeind}$ for $\jtnodeind \in\{1,\dots,\jtsubtreeorder \}$ at Lines~\ref{alg:cta:draw_arb}--\ref{alg:cta:draw_arb_nonempty} (Algorithm~\ref{alg:cta}).
    The neighbors of $\jtnewnode{\jtnodeind}$ for the resulting tree can now be set to be identical to that in $\tree[+]$ by letting $\jtnodestomove{\jtnodeind}=\neig[{\tree[+]}]{\jtnewnode{\jtnodeind}} \setminus \jtnode[\jtnodeind] \setminus \graphnodeset[+]'$ on Line~\ref{alg:cta:draw_random_neighbors}.

    To show \ref{itm:rcta:jt_ind_subgraph_prop} we simply observe that the only vertex removed from $\tree[+]$ compared to $\tree$ is $\gennodea[\jtind+1]$ so that $\trgr(\tree[+])[\{\gennodea[\ell]\}_{\ell=1}^\jtind]=\trgr(\tree)$.

\end{proof}

\begin{proof}[Proof of Theorem \ref{thm:update_mu}]
To reduce some notation we define \(A\eqdef\newseps{\graph}\) and \(A_+\eqdef \newseps{\graph[+]}\).
    Consider the partitions of the separator sets $\sepset(\graph)={A} \cup {A^{\mathsf{c}}}$ and $\sepset(\graph[+])={A_{+}} \cup {A^{\mathsf{c}}_{+}}$.
    In order to show that the factorisation holds it is enough to establish that 
    \begin{enumerate}
        \setlength\itemsep{0.5em}
        \item 
            $\sepset(\graph[+]) = {A_{+}} \cup {A^{\mathsf{c}}}$, \label{itm:update_mu:proof:partition}
        \item 
            $\tree_\sep = {(\tree[+])}_{\sep}$ for $\sep \in {A^{\mathsf{c}}}$, \label{itm:update_mu:proof:ind_tree}
    \end{enumerate}
    where $\tree$ and $\tree[+]$ are arbitrary junction tree representations of the graphs $\graph$ and $\graph[+]$ respectively.
    Note that, showing (\ref{itm:update_mu:proof:partition}) is equivalent to showing that ${A^{\mathsf{c}}_{+}}={A^{\mathsf{c}}}$.

    First let $\sep \in {A^{\mathsf{c}}_{+}}$. 
    It suffice to show that $\sep \in \sepset(\graph)$ since then it follows that $\sep \in {A^{\mathsf{c}}}$. 
    But since $\sep$ is in $\sepset(\graph[+])$ and was not created by the expansion ($\sep \notin \sepset^\star$), 
    it has to come from $\graph$, i.e. $\sep \in \sepset(\graph)$.
    It follows that ${A^{\mathsf{c}}_{+}}\subseteq {A^{\mathsf{c}}}$.

    For the other inclusion, let $\sep\in {A^{\mathsf{c}}}$.
    It suffice to show that $\sep \in \sepset(\graph[+])$ since then it follows that $\sep \in {A^{\mathsf{c}}_{+}}$.
    But if $\sep \in \sepset(\graph)$ and $\sep$ is not subset of any separator in $\sepset(\graph[+])$, it cannot have been removed by the expansion meaning that $\sep \in \sepset(\graph[+])$. 
    Thus, ${A^{\mathsf{c}}} \subseteq {A^{\mathsf{c}}_{+}}$.
    It follows that ${A^{\mathsf{c}}}={A^{\mathsf{c}}_{+}}$.

    To show \ref{itm:update_mu:proof:ind_tree}, let $\sep \in {A^{\mathsf{c}}}$ and consider the tree $\tree_\sep$ spanned by the nodes in $\tree$ associated with separators which are subsets of $\sep$.
    Assume that the tree $({\tree[+]})_\sep$, spanned by the nodes in $\tree[+]$ associated with separators which are subsets of $\sep$, is different from $\tree_\sep$.
    This could only occur in two ways: (i) some new separator $\sep^\star$ that contains $\sep$ has been created or (ii) some separator containing $\sep$ has been removed. However, (i) cannot happen since then $\sep^\star$ would be a new separator in $\sepset^\star$ that would also contain $\sep$ which was not true by assumption. Thus, (ii) must hold. But the only way a separator $\sep'$ of $\tree_\sep$ can be removed is if a new separator $\sep' \cup \{\gennodea[\jtind+1]\}$ also is created. 
    But then $\sep' \cup \{\gennodea[\jtind+1]\}$ would be a new separator in $\sepset^\star$ containing $\sep$, leading to a contradiction.
    This implies that $\jtsepcount_{\graph}(\sep)=\jtsepcount_{\graph[+]}(\sep)$.

\end{proof}

\section{}
\label{appendix:b}
\begin{myalg}[\citet{doi:10.1198/jcgs.2009.07129}] \label{alg:jt_shuffling}

    Given any particular junction tree representation $\tree$, we can choose uniformly at random from the set of equivalent junction trees by applying the following algorithm to the forests $\jtsepforest{\sep}(\tree)$ defined by the distinct separators $\sep$ in $\tree$.
    Following the notation in Theorem~\ref{thm:num_forest}, $r_i$ refers to the size of subtree $i$.
    \begin{enumerate}
     \setlength\itemsep{1em}
        \item 
            Label each vertex of the forest $\{i,j\}$ where $1\le i \le q$ and $1\le j \le r_i$, so that the first index indicates the subtree the vertex belongs to and the second reflects some ordering within the subtree. The ordering of the subtrees and of vertices within subtrees are arbitrary.
        \item 
            Construct a list $v$ containing $q-2$ vertices each chosen uniformly at random with replacement from the set of all $p$ vertices.
        \item 
            Construct a set $w$ containing $q$ vertices, one chosen uniformly at random from each subtree.
        \item 
            Find in $w$ the vertex $x$ with the largest first index that does not appear as a first index of any vertex in $v$. Because the length of $v$ is 2 less than the size of $w$, there must always be at least two such vertices.
        \item 
            Connect $x$ to $y$, the vertex at the head of the list $v$.
        \item 
            Remove $x$ from the set $w$, and delete $y$ from the head of the list $v$.
        \item 
            Repeat from step 4 until $v$ is empty. At this point $w$ contains two vertices. Connect them.
    \end{enumerate}
\end{myalg}